\documentclass{article}
\usepackage{amsmath,amsfonts,amssymb,amsthm}
\usepackage{graphicx}
\usepackage{caption}
\usepackage{subcaption}
\usepackage{natbib}
\usepackage{paralist}
\usepackage{algorithm}
\usepackage{algorithmic}
\usepackage{url}
\usepackage{amsmath}
\usepackage{amsthm}
\usepackage{amsfonts}
\usepackage{xcolor}
\usepackage{booktabs}
\usepackage{siunitx}
\usepackage[hidelinks]{hyperref}
\newcommand{\INPUT}{\item[\textbf{Input:}]}
\newcommand{\OUTPUT}{\item[\textbf{Output:}]}

\DeclareMathOperator*{\argmin}{arg\,min}

\renewcommand{\emptyset}{\varnothing}
\renewcommand{\le}{\leqslant}
\renewcommand{\ge}{\geqslant}

\newcommand{\bsx}{\boldsymbol{x}}
\newcommand{\bsz}{\boldsymbol{z}}

\newcommand{\natu}{\mathbb{N}}

\newcommand{\e}{\mathbb{E}}
\newcommand{\var}{\mathrm{var}}

\newcommand{\vol}{\mathrm{vol}}

\newcommand{\wt}{\widetilde}
\newcommand{\tran}{\mathsf{T}}

\newcommand{\real}{\mathbb{R}}
\newcommand{\ints}{\mathbb{Z}}
\newcommand{\toru}{\mathbb{T}}
\newcommand{\dist}{\mathrm{dist}}

\newcommand{\bsp}{\boldsymbol{p}}

\newcommand{\bsw}{\boldsymbol{w}}
\newcommand{\bsy}{\boldsymbol{y}}
\newcommand{\bsone}{\boldsymbol{1}}

\newcommand{\ca}{\mathcal{A}}
\newcommand{\cp}{\mathcal{P}}
\newcommand{\ce}{\mathcal{E}}
\newcommand{\cek}{\mathcal{E}_k}

\newcommand{\ch}{\mathcal{H}}
\newcommand{\rd}{\,\mathrm{d}}

\newcommand{\runif}{\mathbf{U}}

\newcommand{\bszero}{\boldsymbol{0}}

\newcommand{\bm}{\mathbb{M}}
\newcommand{\mup}{\overline{\bm}}
\newcommand{\mupkmo}{\overline{\bm}_{k-1}}
\newcommand{\mlo}{\underline{\bm}}
\newcommand{\wh}{\widehat}
\newcommand{\cb}{\mathcal{B}}
\newcommand{\cl}{\mathcal{L}}

\newcommand{\co}{\mathcal{O}}
\newcommand{\bdy}{\mathrm{bdy}}
\newcommand{\unpp}{\mathrm{unpp}}
\newcommand{\ma}{\mathrm{MA}}
\newcommand{\bsa}{\boldsymbol{a}}
\newcommand{\bsb}{\boldsymbol{b}}
\newcommand{\bst}{\boldsymbol{t}}
\newcommand{\bse}{\boldsymbol{e}}

\newcommand{\bsu}{\boldsymbol{u}}
\newcommand{\diam}{\mathrm{diam}}

\newcommand{\phm}{\phantom{-}} 

\newtheorem{theorem}{Theorem}

\newtheorem{lemma}{Lemma}

\theoremstyle{definition}
\newtheorem{definition}{Definition}

\newtheorem{remark}{Remark}
\title{Randomized quasi-Monte Carlo for walk on spheres}
\author{Valerie N. P. Ho\\Stanford University \and Art B. Owen \\Stanford University}
\date{July 2026}
\begin{document}
\maketitle
\begin{abstract}
We investigate the use of randomized quasi-Monte Carlo (RQMC) in walk on spheres algorithms to solve boundary value problems for functions with Dirichlet boundary conditions in $\real^d$. For harmonic functions with $d=2$, the integrands of interest are periodic indicator functions over regions $\Theta$ in the torus $\toru^k$.  We give conditions for $\partial\Theta$ to have $k-1$ dimensional Minkowski content which allows us to use results of He and Wang (2015).  The RQMC estimates involve multiple values of $k$. We see sampling variances decreasing with the number $n$ of sample points at slightly better than Monte Carlo rates. The median variance rate in $4$ RQMC methods over $5$ worked examples, including some with $d=3$ and some with nonzero source functions, was slightly better than $O(n^{-1.1})$. The variance reduction factors ranged from $1.8$ to $10.7$ at $n=2^{17}$. None of the four RQMC methods dominated the others.
\end{abstract}
\section{Introduction}

The walk on spheres (WoS) algorithm is a grid free way to solve some differential equations at points $\bsz$ in a closed region $\Omega\subset\real^d$.  It dates back to \cite{mull:1956}, with many early results included in the monograph of \cite{sabe:1991}. It has historically been used less often than the finite element method (FEM) but there has been a resurgence of interest in it, with \cite{sawh:cran:2020} being a notable recent publication. They list several advantages of WoS over FEM, including: WoS can provide a solution at just the most important points in $\Omega$ without requiring one to solve the equations everywhere, WoS saves the cost of generating an FEM grid which can be quite high when $\Omega\in\real^3$ is bounded by millions of triangles, and as a Monte Carlo method, WoS is easily parallelized. 

In this paper we will use randomized quasi-Monte Carlo (RQMC) in WoS computations.  There was earlier work using quasi-Monte Carlo (QMC) to solve WoS problems by \cite{masc:hwan:2003} and \cite{masc:kara:hwan:2004}.  The WoS integrands typically have infinite variation in the sense of Hardy and Krause. We say that they are `not in BVHK'. This infinite variation complicates the theoretical treatment of their QMC error. In RQMC, by contrast, we only need the integrand to be in $L_2$ in order to ensure a mean squared error of $o(1/n)$ on $n$ function evaluations.

What we typically see in our examples is a mean squared error (MSE) for RQMC that, for $n$ sample points, decreases in proportion to a power of $n$ slightly below $-1$, commonly comparable to $n^{-1.1}$ or $n^{-1.2}$.  A rate of $n^{-1.2}$ can yield a variance reduction of five to ten fold without using extremely large sample sizes.  The WoS integrands are discontinuous. It is common for them to be indicator functions of sets that do not have axis parallel boundaries, which is why they are not in BVHK.  To explain the variance reductions we see, we draw on and extend some work of \cite{he:wang:2015} and \cite{liu:2025}. These papers show improved convergence rates when the integrand is the indicator function of a set $\Theta\subset[0,1)^d$ whose boundary $\partial\Theta$ has finite $d-1$ dimensional Minkowski content, or when it is a sufficiently regular function multiplied by such an indicator. For a $d$ dimensional integrand they get a variance of $O(n^{-1-1/d})$. A major part of this paper is devoted to giving sufficient conditions for the sets in WoS to have finite Minkowski content. In our setting, the integrand also has unbounded dimension, though it is commonly truncated to finite dimension.

The path to finite Minkowski content is long because even if $\partial\Omega$ is made up of infinitely differentiable curves there can be severe pathologies as shown by \cite{choi:choi:moon:1997} whose work we describe below.  They rule out some of those pathologies using real analytic curves.  In our analysis of the WoS setting, the distance to subsets of $\partial\Omega$ is critically important. Under our conditions, that function will be piecewise $C^\infty$ but it will only be Lipschitz globally and this complicates some of the manifolds that we study.

This paper is organized as follows.
Section~\ref{sec:background} gives a basic account of WoS to orient the reader and a brief note on the RQMC methods we use. Section~\ref{sec:gasket} illustrates the use of RQMC on the WoS problem for a gasket example taken from \cite{wosoneweekend}.  There we see that RQMC methods have an MSE that decays faster than $n^{-1}$ for $n\le 2^{17}$.
Section~\ref{sec:axes} provides background on Minkowski content, rectifiability and the medial axes of domains in $\real^2$.  Section~\ref{sec:diffgeo} lists some theorems from differential geometry that we need. Section~\ref{sec:main} has our main theorem. Letting $\Theta_k\subset[0,1)^k$ represent the set of RQMC points that cause WoS to reach a given portion of $\partial\Omega\subset\real^2$ in exactly $k$ steps, we give sufficient conditions for $\partial \Theta_k$ to have finite $k-1$ dimensional Minkowski content.
Section~\ref{sec:rqmcforwos} discusses how the main theorem applies to RQMC problems.
Section~\ref{sec:moreexamples} briefly describes more numerical examples including three dimensional domains that are not covered by our theorems. For four RQMC methods and five examples we see modest variance reduction factors ranging from $1.8$ to $10.7$ at $n=2^{17}$. The code used for our numerical examples is publicly available at
\url{https://github.com/hoval58/RQMC-WoS}.
Section~\ref{sec:discussion} has some final comments.

\section{Background on WoS and RQMC}\label{sec:background}

Here we describe WoS and RQMC.  For an (R)QMC readership we present WoS in a way that shows that a random walk algorithm is very natural and then we use RQMC in that random walk. For this readership we only emphasize a few RQMC details that are specific to this project.

Here is some notation that we use throughout the paper. The positive integers are denoted by $\natu$ and $\natu_0=\natu\cup\{0\}$.
For a set $\Theta\subset\real^d$ we use $\Theta^\circ$ for its interior, $\overline{\Theta}$ for its closure and $\partial\Theta$ for its boundary. If $\Theta$ is Lebesgue measurable, then $\vol(\Theta)$ is its measure.
For nonempty $\Theta\subset\real^d$ and $\bsz\in\real^d$ we define the distance function
$$\dist(\bsz,\Theta) = \inf_{\tilde\bsz\in\Theta}\Vert\bsz-\tilde\bsz\Vert.$$
For fixed $\Theta$, $\dist(\bsz,\Theta)$ is a Lipschitz function of $\bsz$. If $\Theta$ is closed, then the infimum is attained, i.e., $\dist(\bsz,\Theta)=\Vert\bsz-\bar\bsz\Vert$ where $\bar\bsz\in\Theta$ is a (not necessarily unique) projection of $\bsz$ onto $\Theta$. We use $\bsone_\Theta$ for the function taking the value $1$ on $\Theta$ and $0$ on $\Theta^c$.

For $\varepsilon\ge0$, we will use balls and spheres
$B(\bsz,\varepsilon) =\{\tilde\bsz\in\real^d\mid \Vert\tilde\bsz-\bsz\Vert\le\varepsilon\}$ and $S(\bsz,\varepsilon)=\{\tilde\bsz\in\real^d\mid\Vert\tilde\bsz-\bsz\Vert =\varepsilon\}=\partial B(\bsz,\varepsilon)$, respectively. We follow \cite{choi:choi:moon:1997} in allowing $\varepsilon=0$, and then $B(\bsz,0)=S(\bsz,0)=\{\bsz\}$. When we need to indicate the dimension, then we use $B_d$ and $S_d$.

\subsection{Basic walk on spheres}
The problem in WoS is to solve a differential equation at some point $\bsz_0$ inside a bounded domain $\Omega\subset\real^d$.  In the simplest setting, we want to compute $u(\bsz_0)$ for a harmonic function $u:\Omega\to\real$ at a point $\bsz_0\in\Omega^\circ$, and we are given a boundary function $h$, so that $u(\bsz)=h(\bsz)$ for $\bsz\in\partial\Omega$. This is the Dirichlet boundary condition. The Neumann boundary condition specifies the outward normal derivative on $\partial \Omega$, and hybrid combinations of these two types also exist. However, in this paper we only consider boundary value problems (BVPs) with a Dirichlet condition.

A harmonic function $u$ has Laplacian $\Delta u(\bsz) =\sum_{j=1}^d\partial^2u(\bsz)/\partial z_j^2=0$.  For standard coordinate vectors $\bse_j$, writing
$$
0=\Delta u(\bsz_0)\approx \frac1{\varepsilon^2}\sum_{j=1}^d
\bigl(u(\bsz_0+\varepsilon \bse_j)-2u(\bsz_0)+u(\bsz_0-\varepsilon\bse_j)\bigr)$$
shows that $u(\bsz_0)$ is approximately the average of $u$ at $2d$ neighbors.  We could pick one of those neighbors (call it $\bsz_1$) at random and then $u(\bsz_1)$ is a nearly unbiased estimate of $u(\bsz_0)$. The value $u(\bsz_1)$ is nearly the average of its $2d$ neighbors and we could keep randomly selecting neighbors until we get $\bsz_k$ in (or very close to) $\partial\Omega$ where $u$ is known. Now $u(\bsz_k)$ is a nearly unbiased estimate of $u(\bsz_0)$. 

As $\varepsilon\to0$ the random walk converges to a Brownian motion in $\Omega$ that first hits $\partial\Omega$ at some time $\tau>0$. Then $u(\bsz_\tau)$ is an unbiased estimate of $u(\bsz_0)$. That is, $u(\bsz_0)=\e(h(\bsz_\tau))$ where $\tau = \inf\{t\in[0,\infty)\mid \bsz_t\in\partial\Omega\}$.

The key insight in WoS is that when a Brownian motion exits the ball $B(\bsz_k,r)$ for $r>0$, it exits with a uniform distribution over $\partial B(\bsz_k,r)$. Instead of generating a Brownian motion starting at $\bsz_0$, we can simply take $\bsz_{k}\sim\runif(S(\bsz_{k-1},r_k))$ for integers $k\ge1$ where $r_k=\dist(\bsz_{k-1},\partial\Omega)$.

For most geometries, the WoS algorithm will never produce $\bsz_k\in\partial\Omega$. Instead we stop the WoS when $\bsz_k\in\partial\Omega_{\varepsilon}:=\{\bsz\in\Omega\mid \dist(\bsz,\partial\Omega)<\varepsilon\}$. This happens at step $\tau = \min\{k\in\natu_0\mid \bsz_k\in\partial\Omega_{\varepsilon}\}$. We then return the value $h(\bar\bsz_{\tau})$ where for $\bsz\in\Omega$, $\bar\bsz$ is the projection of $\bsz$ onto $\partial\Omega$ with an arbitrary tie-breaker rule when that projection is not unique. The $i$-th run in an MC approach to WoS generates a trajectory $\bsz_{i,k}$ for $1\le k\le \tau_i$ starting at $\bsz_{i,0}=\bsz_0$. Then a Monte Carlo WoS estimate of $u(\bsz_0)$ is
\begin{align}\label{eq:mcwos}
\hat u(\bsz_0)=\frac1n\sum_{i=1}^n h(\bar \bsz_{i,\tau_i}).
\end{align}

A recent innovation of \cite{czek:etal:2024} pools information over a set of nearby starting points.  When $\tilde\bsz_0\in B(\bsz_0,r_1)$, the authors apply an importance sampling weight to the distribution of $\bsz_1$ and then reuse the entire walk for $\bsz_0$ in their estimate of $u(\tilde\bsz_0)$. They do not need to apply weights for $\bsz_k$ with $k\ge2$.

\subsection{Bias and running time}
A WoS algorithm stops when $\dist(\bsz_k,\partial\Omega)<\varepsilon$. This raises questions about how the bias from stopping short of the boundary and how the number of steps to reach the boundary both depend on $\varepsilon$.
\cite{bind:brav:2012} give a very complete analysis of the number of steps to reach $\partial\Omega\subset\real^d$.  They have a notion of $\alpha$-thickness for $\alpha\in[0,d]$ where $\alpha>0$ allows for fractal behavior of $\partial\Omega$. 
For any $\alpha<2$, the expected number of steps to get close to $\partial\Omega$ is $O(\log(1/\varepsilon))$ as $\varepsilon\to0$. For $\alpha=2$ the rate is $O( \log^2(1/\varepsilon))$. For $\alpha>2$ it is $O((1/\varepsilon)^{2-4/\alpha})$. The domains in our examples are all $0$-thick. That can be verified by applying Definition 1 of \cite{bind:brav:2012}. As a result, the number of steps taken are all $O( \log(1/\varepsilon))$. 
In addition to a bound on the expected number of steps, they also consider the tail probability. Their Remark $3$ gives an exponential decay for the probability that more than $k$ steps are required.

The present understanding of the bias in WoS is less well developed than the convergence time. There is extensive empirical work in \cite{masc:hwan:2003} where the bias was seen to be $O(\varepsilon)$ in every one of a range of examples.  They reason that the problem has distorted $\partial\Omega$ by $O(\varepsilon)$ and in that case a bias of $O(\varepsilon)$ is natural. \cite{sawh:cran:2020} report on the small bias in WoS in some large scale computations.

\subsection{Nonzero source term}\label{subsec:source}

A more general boundary value problem has $u(\bsz)$ known for $\bsz\in\partial\Omega$ and subject to $\Delta u(\bsz)=g(\bsz)$ for $\bsz\in\Omega$.
When the source term $g$ is nonzero, then the solution at $\bsz\in\Omega^\circ$ involves some Green's functions $G_d^B$ that we define below. We follow the derivation in \cite{sawh:mill:gkio:cran:2023}. They use $\Delta$ to represent the negative semi-definite Laplacian, so their $\Delta u$ is our $-\Delta u$. In order to use their presentation we replace our source function $g$ by $-g$. The equation (11) of \cite{sawh:mill:gkio:cran:2023} gives us
\begin{align*}
u(\bsz_0) = \frac1{\vol(S(\bsz_0,r_1))}\int_{S(\bsz_0,r_1)}u(\bsz)\rd\bsz + \int_{B(\bsz_0,r_1)}G_d^{B(\bsz_0,r_1)}(\bsz_0,\bsz)(-g)(\bsz)\rd\bsz.
\end{align*}
That justifies the estimate
\begin{align}\label{eq:woswithsource2}
\hat u(\bsz_0)
=
h(\bar{\bsz}_{\tau})
-\sum_{k=1}^{\tau}
\vol(B(\bsz_{k-1},r_k))\,
G_d^{B(\bsz_{k-1},r_k)}(\bsz_{k-1},\bsw_k)\, g(\bsw_k)
\end{align}
that we average over $n$ independent replicates using independent vectors $\bsw_k\sim \runif(B(\bsz_{k-1},r_k))$ and $\bsz_{k}\sim \runif(S(\bsz_{k-1},r_k))$ each time.

Our examples have $d\in\{2,3\}$ and for $\bsw\in B(\bsz,r)$, Appendix A.1 of \cite{sawh:mill:gkio:cran:2023} gives 
$$
G_2^{B(\bsz,r)}(\bsz,\bsw)=
\frac{\log(r\Vert\bsw-\bsz\Vert^{-1})}{2\pi}
\quad\text{and}\quad
G_3^{B(\bsz,r)}(\bsz,\bsw)=
\frac1{4\pi}\Bigl(\frac1{\Vert\bsw-\bsz\Vert}-\frac1r\Bigr).
$$
Our $r$ is their $R$ and their $r$ is our $\Vert \bsw-\bsz\Vert$.

An interesting special case arises when the source term $g$ takes a constant value $\nu$. Then instead of sampling $\bsw_k$, 
we may use
\begin{align}\label{eq:wosconstsource}
\hat u(\bsz_0)
&=h(\bar{\bsz}_{\tau})
-\nu\sum_{k=1}^{\tau}
\int_{B(\bsz_{k-1},r_k)}
G_d^{B(\bsz_{k-1},r_k)}(\bsz_{k-1},\bsw)\rd\bsw\notag\\
&=h(\bar{\bsz}_{\tau})
-\frac{\nu}{2d}\sum_{k=1}^{\tau}
r_k^2
\end{align}
for $d=2,3$ by equations (27) and (28) of \cite{sawh:mill:gkio:cran:2023}, with $\bsz_{k}\sim \runif(S(\bsz_{k-1},r_k))$. This simplifies further when $u$ vanishes on $\partial\Omega$ which we see for a dumbbell-shaped domain in Section~\ref{sec:dumbbell}.

\subsection{Randomized Quasi-Monte Carlo}\label{sec:rqmc}

We assume familiarity with the basic notions of QMC and RQMC. This background can be found in \cite{nied:1992}, \cite{dick:pill:2010} or \cite{practicalqmc}.  

We use RQMC points $\bsx_1,\dots,\bsx_n\in[0,1]^d$ to approximate $\mu = \int_{[0,1]^d}f(\bsx)\rd\bsx$ by $(1/n)\sum_{i=1}^nf(\bsx_i)$. The function $f$ includes the transformation of uniform random variables to the WoS trajectories and sample points $\bsw$ as well as the evaluations of $h$ and $g$. The RQMC points individually satisfy $\bsx_i\sim\runif[0,1]^d$ and that gives $\e(\hat\mu)=\mu$. They are more evenly distributed through $[0,1]^d$ than IID Monte Carlo (MC) points would be as quantified by discrepancy measures, which is a reason to believe that they will provide more accuracy than MC points do.

We include several RQMC algorithms from QMCPy \citep{QMCPy}. These are scrambled Sobol' points, scrambled Niederreiter points, scrambled Halton points and randomly shifted lattice rules.  The Sobol' points use the direction numbers from \cite{joe:kuo:2008} and are given the matrix scramble and digital shift from \cite{mato:1998:2}.  The Niederreiter points and the Halton points are similarly given those scrambles.   The default lattice rule in QMCPy at the time of our computation is lattice-33002-1024-1048576.9125 from Frances Kuo's site:
\url{https://web.maths.unsw.edu.au/~fkuo/lattice/}. That lattice was designed for sample sizes $n=2^m$ for $10\le m\le 20$.  We got good results for smaller values of $n$ but anybody getting poor results for $n<2^{10}$ should not view that as a weakness of the lattice.

The scrambled Sobol', Niederreiter and Halton points all satisfy $\var(\hat\mu) = o(1/n)$ as $n\to\infty$ while even an adversarially chosen integrand in $L_2$ yields $\var(\hat\mu)\le\Gamma_d\sigma^2/n$ for finite $n$ and some gain coefficient $\Gamma_d<\infty$ where $\sigma^2/n$ is the variance we would have had using Monte Carlo (MC) instead of QMC. The value of $\Gamma_d$ grows only logarithmically in $d$ for the scrambled Halton points \citep{owen:pan:2024} while it may grow exponentially in $d$ for scrambled nets.
There is also a strong law of large numbers for scrambled nets such as those of Sobol' or Niederreiter, when $f\in L_{1+\varepsilon}$ \citep{sllnrqmc}.

We will show numerical results for all these methods. The RQMC rules typically all outperform MC and have very nearly the same performance as each other. There was one example where the Niederreiter points performed notably worse than the other methods.

While lattice rules do not have a bounded gain coefficient $\Gamma$, 
we took a special interest in them because WoS integrands are periodic and lattice rules are especially well suited to periodic integrands \citep{sloa:joe}. The best results for lattice rules also depend on smoothness, but the WoS integrands are not smooth.

We will need some results from \cite{he:wang:2015} on the accuracy of RQMC estimation for integrands $f(\bsx) =\bsone_\Theta(\bsx)$  and the extensions in \cite{liu:2025} to $\bsone_\Theta(\bsx)g(\bsx)$ where $\Theta\subset[0,1]^k$ is measurable and $g$ satisfies a boundary condition. Note that this $g$ is distinct from our source function.
The function $\bsone_{\Theta}$
is typically not in BVHK \citep{variation}. It is of course in $L_2$ and so RQMC has variance $o(1/n)$ for it but we will quote some sharper rates.  Those results use the notion of Minkowski content which we present in Section~\ref{sec:axes} along with other geometric definitions.

Theorem 3.5 of \cite{he:wang:2015} shows that if $g$ is in BVHK and $\partial\Theta$ has finite $k-1$ dimensional Minkowski content, then scrambled net integration of $\bsone_\Theta g$ has a mean squared error of 
$O(n^{-1-1/(2k-1)}\log(n)^{2(k-1)/(2k-1)})$. Theorem 4.4 of that same paper shows that the scrambled net integration error of $\bsone_\Theta$ in that case is $O(n^{-1-1/k})$. Their upper bounds include leading constants. \cite{liu:2025} shows that the rate is $O(n^{-1-1/k})$ when 
$$\biggl|\frac{ \partial^{|u|}g(\bsz)}{\prod_{j\in u}\partial z_j}\biggr|
\le \prod_{j=1}^k\min(z_j,1-z_j)^{-A_j-1_{\{j\in u\}}}
$$
holds for all non-empty $u\subseteq\{1,2,\dots,k\}$ and $\max_jA_j<1/2$. In particular this holds when all the above partial derivatives of $g$ are bounded.

\subsection{The WoS integrand in RQMC}\label{sec:wosinrqmc}

Each step of the WoS algorithm consumes some number $s\ge1$ of uniform random variables.  Let $\psi_0:[0,1)^{s_0}\to S(\bszero,1)$ with $\psi_0(\bsz)\sim\runif(S(\bszero,1))$ when $\bsz\sim\runif([0,1)^{s_0})$ and $\psi_1:[0,1)^{s_1}\to B(\bszero,1)$ with $\psi_1(\bsz)\sim\runif(B(\bszero,1))$ when $\bsz\sim\runif([0,1)^{s_1})$. For harmonic $u$, we can then use $s=s_0$ uniform variables to take one step of WoS.  The first step is
$\bsz_1 \gets \bsz_0 + r_1\psi_0(\bsx)$ for $\bsx\in[0,1)^{s_0}$ and to allow for $K$ steps requires $Ks_0$ uniform variables. We then use an RQMC point set with $\bsx_i\in[0,1)^{Ks_0}$ for $i=1,\dots,n$.
When the source $g$ is nonzero, we use $s=s_0+s_1$ uniform variables to take the step and also sample $\bsw_k$.  It then requires RQMC points $\bsx_i\in[0,1)^{K(s_0+s_1)}$ to generate the WoS estimate.  The RQMC point sets we use are all high dimensional and support very large $K$ and the WoS convergence is fast enough that most samples have small $\tau$.

There are standard mappings for $\psi_0$ and $\psi_1$. The ones in \cite{fang:wang:1994} have $s_0=d-1$ and $s_1=d$.  Our examples in this paper have $d=2$ or $3$.

For $d=2$, we always take $\psi_0(x) = \theta(x) := (\cos(2\pi x),\sin(2\pi x))\sim\runif(S_2(\bszero,1))$ for $x\sim\runif[0,1)$. 
We use $\psi_1(\bsx)=(\sqrt{x_1}\cos(2\pi x_2),\sqrt{x_1}\sin(2\pi x_2))$ for $\bsx\sim\runif([0,1)^2)$ to sample $\runif(B_2(\bszero,1))$.

For $d=3$, we sample $\bsx\sim\runif([0,1)^2)$, define a latitude $\lambda(\bsx)=2x_1-1$, a longitude $\theta(\bsx)= 2\pi x_2$ and use the hat box transformation
\begin{align}\label{eq:hatbox}
\psi_0(\bsx)=\bigl(\sqrt{1-\lambda^2(\bsx)}\cos(\theta(\bsx)),\,\sqrt{1-\lambda^2(\bsx)}\sin(\theta(\bsx)),\,\lambda(\bsx)\bigr)
\end{align}
to get $\psi_0(\bsx)\sim\runif(S_3(\bszero,1))$. None of our examples with $d=3$ had a nonzero source function, but for that case we would use $\bsw=\psi_1(\bsx) = x_1^{1/3}\psi_0((x_2,x_3))$ for $x\sim \runif([0,1)^3)$.

The case of a harmonic function on $\Omega\subset\real^2$ is especially interesting as $\bsz_k$ is now a periodic function of the point $\bsx\in[0,1)^k$ from which the first $k$ WoS steps are generated. As mentioned earlier, periodic functions are especially favorable for lattice sampling.

Algorithm \ref{alg:wos} has pseudo-code to compute the WoS estimate of $u(\bsz_0)$ by RQMC when the source term $g=0$.  For a nonzero source term, the algorithm should be extended to include a second mapping $\psi_1$ to generate the point $\bsw_k$ at step $k$ from $s_1$ uniform inputs, and the RQMC points should have dimension $K(s_0+s_1)$, with $K$ being the maximum number of steps the WoS can take per trajectory.

\begin{algorithm}[t]
\caption{Walk on Spheres with RQMC\label{alg:wos} for solving a Dirichlet BVP}
\begin{algorithmic}[1]
\INPUT
\begin{tabular}[t]{@{}l l@{}}
- & Initial point $\bsz_0 \in \Omega\subset\real^d$ \\
- &termination parameter $\varepsilon > 0$, maximum \# steps $K$\\
- & RQMC points $\bsx_1,\dots,\bsx_n\in[0,1)^{sK}$\\
- & A mapping $\psi_0$ such that $\psi_0(\bsu)\sim \runif(S_d(\bszero,1))$ for $\bsu\sim \runif([0,1)^s)$
\end{tabular}
\newline
\OUTPUT Estimator $\hat{u}_n(\bsz_0)$ of $u(\bsz_0)$.
\newline
\STATE Initialize $\hat\mu \leftarrow 0$
\FOR{$i=1,\dots,n$}
    \STATE $\bsz \leftarrow \bsz_0$
    \FOR{$k=1,\dots,K$}
        \STATE $r \leftarrow \dist(\bsz,\partial\Omega)$
        \IF{$r < \varepsilon$}
            \STATE Compute the projection $\bar{\bsz}$ of $\bsz$ onto $\partial\Omega$
            \STATE $Y_i \leftarrow h(\bar{\bsz})$
            \STATE \textbf{break}
        \ENDIF
        \STATE Extract the block $\bsu_{i,k}=\bsx_{i,s(k-1)+1{:}sk} \in [0,1)^s$ from $\bsx_i$
        \STATE $\bsz \leftarrow \bsz + r\psi_0(\bsu_{i,k})$
    \ENDFOR
    \IF{$k=K$}
        \STATE Compute the projection $\bar{\bsz}$ of $\bsz$ onto $\partial\Omega$
        \STATE $Y_i \leftarrow h(\bar{\bsz})$
    \ENDIF
    \STATE $\hat\mu \leftarrow \hat\mu + Y_i$
\ENDFOR
\RETURN $\hat{u}_n(\bsz_0) \leftarrow \hat\mu/n$
\end{algorithmic}
\end{algorithm}

\section{The gasket example}\label{sec:gasket}

We illustrate RQMC-WoS with an example from \cite{wosoneweekend}. The authors consider the thermal conduction in a \textit{cylinder head gasket} illustrated in Figure~\ref{fig:gasket}. In this problem $u(\bsz)$ is the temperature of the gasket at $\bsz\in\Omega$. The purpose of the head gasket is to avoid any leakage of oil or coolant from the combustion engine into the cylinders where combustion occurs. Leaks are more likely to occur where the gasket is hotter. 

They are interested in approximating the solution to the  BVP with $\Delta u(\bsz)=0$ for $\bsz\in\Omega^\circ$
and boundary values
\begin{equation*}
u(\bsz) = \sum_{r\in\{\mathrm{coolant},\,\mathrm{outer},\,\mathrm{oil},\,\mathrm{oil\ return},\,\mathrm{bore}\}}
T_r\,\bsone_r(\bsz) \quad \textrm{ for } \bsz\in\partial\Omega
\end{equation*}
where $T_r$ is a constant temperature (in degrees Celsius) specific to boundary component $r$ and $\bsone_r(\bsz)$ is the indicator of component $r$.
We computed the WoS estimator of $u$ at point $\bsz_0 =(0.240999,0.3)$ located above the center of the third borehole and roughly half way to the edge of the gasket as indicated in Figure~\ref{fig:gasket}. This point is sufficiently interior to avoid immediate absorption, yet close enough to several boundary components that typical WoS trajectories require multiple steps and the exit location is nontrivial. 
For each sample size, we consider Monte Carlo sampling and the four RQMC algorithms described in Section~\ref{sec:rqmc}, and run $100$ independent replicates of each method. 
The variance curves are shown in Figure~\ref{fig:gasket-results}. 
The reference curve was fit by a regression of log variance on $\log(n)$, pooling data from all four RQMC methods for $n\ge 2^7$. We use that definition for all the reference curves in this paper. The regression coefficients for individual RQMC methods are given in Table~\ref{tab:rqmc_slopes_intercepts} of Section~\ref{sec:summary} for five numerical examples.

\begin{figure}
    \centering
    \includegraphics[width=1\linewidth]{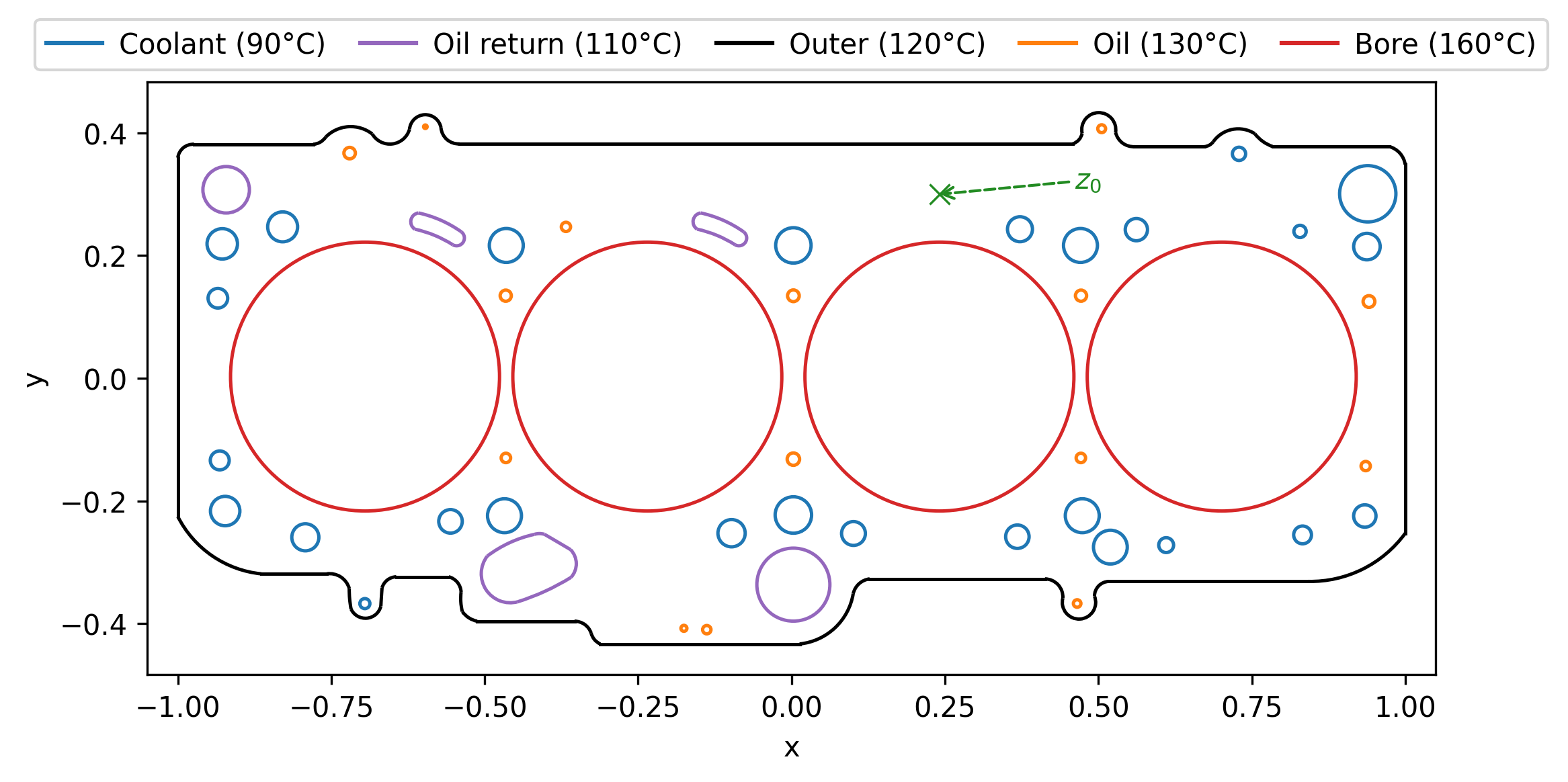}
    \caption{Cylinder head gasket domain. The domain $\Omega$ is the interior of the outer curve minus some shapes for coolant, oil return, oil and a bore.}
    \label{fig:gasket}
\end{figure}

\begin{figure}[t]
\centering\includegraphics[width=.9\hsize]{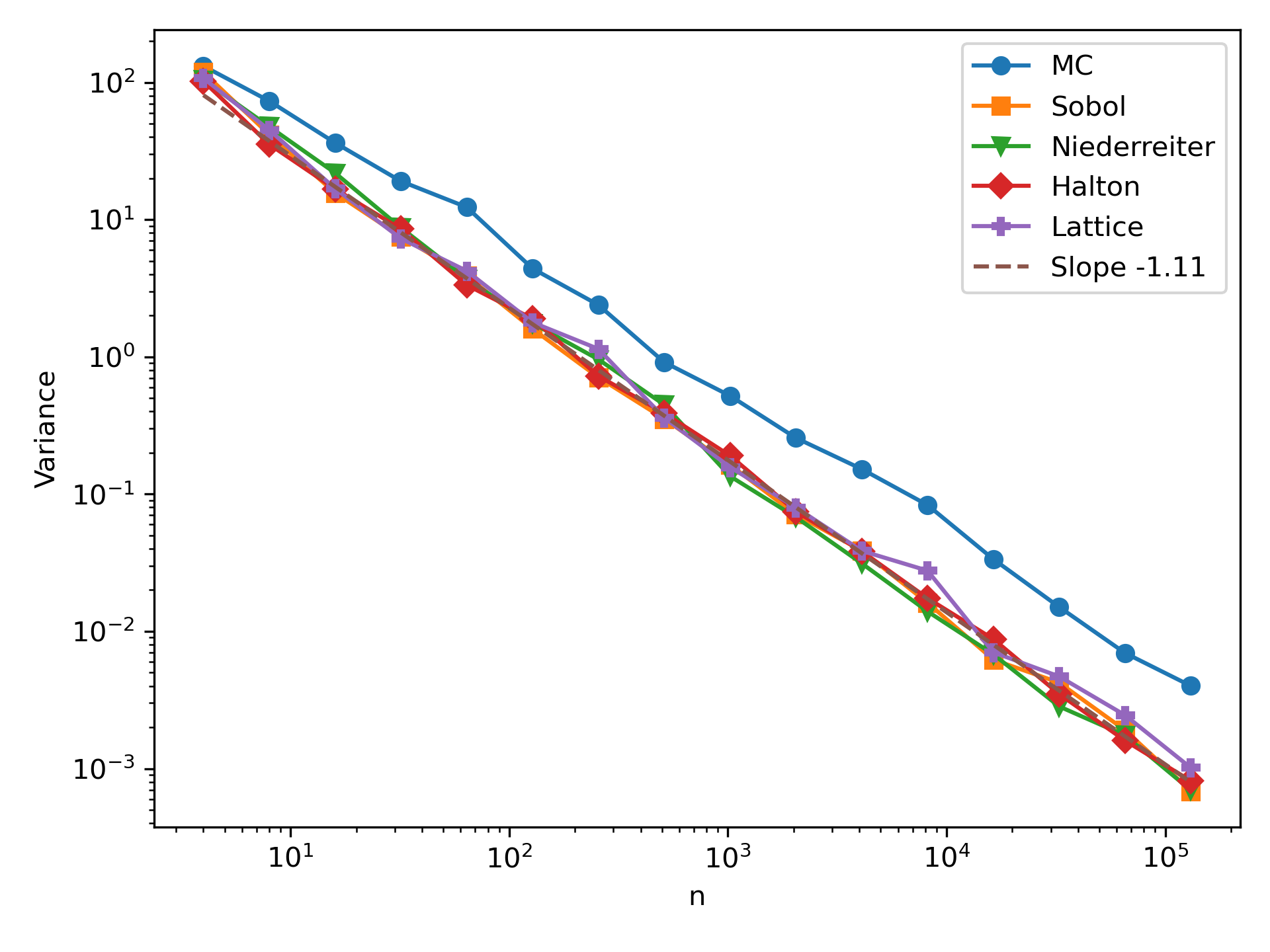}
    \caption{Variance of the standard WoS estimator at $\bsz_0=(0.240999, 0.3)$ for the gasket domain with $\varepsilon=10^{-3}$. 
    }
    \label{fig:gasket-results}
\end{figure}

In this example we see that using RQMC points brings a better convergence rate for the variance, with a slope near $-1.11$ in a log-log plot.  As mentioned earlier, the different RQMC samplers have very nearly equal performance. Here, for $n=2^{17}$, the RQMC points bring a roughly $5$-fold variance reduction, compared to standard MC. 
The performance of RQMC varies spatially in this domain.  We describe some other results in Section~\ref{sec:moreexamples}.

\section{Geometric quantities}\label{sec:axes}

It seems intuitively clear that the sets of RQMC points
that lead to the outcomes of interest in WoS problems
should be quite regular and not, for instance, of a fractal
nature. However, finding sufficient conditions to justify
this belief runs into some complications.
Here, we present some background results
on Hausdorff measure, Minkowski content, rectifiability of curves,
domains in $\real^2$ and analyticity of distances to plane curves.

\subsection{Hausdorff measure and Minkowski content}

The way we will control Minkowski content is via Theorem 3.2.39 of \cite{fede:1969} quoted in Section~\ref{sec:rectifiability} which gives conditions that make it equal to Hausdorff measure. Then we will apply theorems that control Hausdorff measure.

We use the definition of Hausdorff measure from \cite{evan:gari:2015} because it is more easily stated than some others.  An arbitrary set $C\subset\real^k$ has diameter $\diam(C)=\sup_{\bsx,\tilde\bsx\in C}\Vert\bsx-\tilde\bsx\Vert$, with $\diam(\emptyset)=0$ by convention. Then for $\Theta\subset\real^k$, $0\le m<\infty$ and $0<\delta\le\infty$ we let
$$
\ch_{m,\delta}(\Theta) = \frac{\pi^{m/2}}{\Gamma(m/2+1)}\times
\inf\Biggl\{\sum_{\ell=1}^\infty  \Bigl(\frac{\diam(C_\ell)}2\Bigr)^m
\biggm| \Theta\subset\bigcup_{\ell=1}^\infty C_\ell, \diam(C_\ell)\le\delta
\Biggr\}
$$
where $\Gamma(\cdot)$ is the Gamma function.
When $m$ is an integer the leading factor above is $\vol(B_m(\bszero,1))$.
The $m$-dimensional Hausdorff measure of $\Theta$ is
\begin{align}\label{eq:defhaus}
\ch_m(\Theta) = \lim_{\delta\to0}\ch_{m,\delta}(\Theta)
=\sup_{\delta>0}\ch_{m,\delta}(\Theta).
\end{align}
We only need $\ch_m$ for integer $m$. For $m=0,1,2,3$, we get a measure that matches the usual notions of cardinality, length, area and volume. The thickness $\alpha$ in \cite{bind:brav:2012} is defined using a closely related quantity called Hausdorff content.

We follow Chapter 3.3 of \cite{kran:park:1999} to define Minkowski content.
For $m\in\{0,1,\dots,k\}$,
the upper $m$-dimensional Minkowski content of non-empty $\Theta\subset\real^k$ is
$$
\mup_{m}(\Theta) = \limsup_{\varepsilon\to0^+} 
\frac{\vol(\{\bsx\mid \dist(\bsx,\Theta)<\varepsilon\})}{\vol(B_{k-m}(\bszero,\varepsilon))}.
$$
The corresponding lower $m$-dimensional Minkowski content is
$$
\mlo_{m}(\Theta) 
= \liminf_{\varepsilon\to0^+} 
\frac{\vol(\{\bsx\mid \dist(\bsx,\Theta)<\varepsilon\})}{\vol(B_{k-m}(\bszero,\varepsilon))}.
$$
When $\mlo_m(\Theta)=\mup_m(\Theta)$, we let $\bm_m(\Theta)$
denote their common value.  If $\bm_m(\Theta)<\infty$, then we say that $\Theta$ has $m$ dimensional Minkowski content.

The way Minkowski content is used in RQMC is to partition
the set $[0,1)^k$ into $N$ sets $T_1,\dots,T_N$.  We write our integrand as 
$ f(\bsx) = 1\{\bsx\in\Theta\} = \sum_{j=1}^Nf_j(\bsx)$
where $f_j(\bsx) = 1\{\bsx\in \Theta\cap T_j\}$. Then our estimate of $\vol(\Theta)$ is
$$
\wh\vol(\Theta) = \frac1n \sum_{i=1}^n\sum_{j=1}^Nf_j(\bsx_i)
$$
for $n$ sample points $\bsx_i\in[0,1)^k$.

For scrambled $(t,m,k)$-nets in base $b$ it is convenient to take
$T_j$ to be elementary intervals of volume $b^{t-m}$ with side lengths
as nearly equal as possible.
For the commonly used Sobol' points with $b=2$, these are dyadic hyperrectangles.
Then $(1/n)\sum_{i=1}^nf_j(\bsx_i) = \int_{[0,1)^k}f_j(\bsx)$ if
either $T_j\subset\Theta$ or $T_j\subset\Theta^c$.
Let
 $$\cb=\bigl\{ j\in\{1,2,\dots,N\}\mid T_j\cap \Theta\ne\emptyset\ \&\ T_j\cap \Theta^c\ne\emptyset\bigr\}$$
be the set of `boundary' elementary intervals.
Then
$\wh\vol(\Theta) -\vol(\Theta)$
is the error in estimating the integral of
$$
f_\bdy(\bsx)=\sum_{j\in\cb}f_j(\bsx).
$$
The RQMC variance is no more than some gain factor $\Gamma<\infty$ times
the Monte Carlo variance. For $n\ge2$, that Monte Carlo variance is no more than $|\cb|/n^2$
because $f_j(\bsx)\in\{0,1\}$ takes the value $1$ with probability at most $1/n$
giving it variance at most $1/n$. Then the RQMC variance is at most $\Gamma|\cb|/n^2$. 
If $\mupkmo(\partial\Theta)<\infty$, then $|\cb| = O(N^{(k-1)/k})$
and for digital nets $N = n/b^t$.  Then the RQMC variance is $O(n^{-2+(k-1)/k})=O(n^{-1-1/k})$.  

The RQMC argument only needs a finite upper Minkowski content for $\partial\Theta$.
Because the RQMC points
are uniformly distributed we could reduce the boundary set to 
\begin{align*}
\tilde\cb=\bigl\{ j\in\{1,2,\dots,N\}\mid 
0<\vol(T_j\cap\Theta)<\vol(T_j)\bigr\}.
\end{align*}
We are not aware of any proofs in the literature
that use $\tilde\cb$ instead of $\cb$. 

\subsection{Rectifiability from Federer (1969)}\label{sec:rectifiability}

Here are the definitions from Section 3.2.14 of \cite{fede:1969}. 
He references a measure $\phi$ that is commonly taken to be $m$ dimensional Hausdorff measure. 
In each definition $m$ is a positive integer and $E$ is a subset of 
a metric space that we take to be $\real^d$ for some $d\ge1$. For $\phi$ to measure $\real^d$ in Federer's sense, it must be a countably subadditive function on all subsets of $\real^d$, what we now call an outer measure. 

\begin{definition}\label{defn:rect}
\leavevmode 
\begin{compactenum}[\quad(1)]
\item $E$ is $m$ rectifiable if and only if 
there is a Lipschitz function mapping a bounded subset of $\real^m$
onto $E$. 
\item $E$ is countably $m$ rectifiable if and only if 
$E$ equals the union of some countable family 
of $m$ rectifiable sets. 
\item $E$ is countably $(\phi,m)$ rectifiable if and only if 
$\phi$ measures $\real^d$ and 
there is a countably $m$ rectifiable set containing $\phi$
almost all of  $E$. 
\item $E$ is $(\phi,m)$ rectifiable if and only if $E$ is countably 
$(\phi,m)$ rectifiable and $\phi(E)<\infty$. 
\end{compactenum}
\end{definition}

We will need to make a somewhat subtle use of Definition~\ref{defn:rect}
part (1). The bounded subset of $\real^m$ that we use might
have strictly lower dimension than $m$.
We will use the fact that any subset $S$ of an $m$ rectifiable 
set $E$  is also $m$ rectifiable, even for a set $S$ with a
fractal appearance.  This holds because we can write $S = f(f^{-1}(S)\cap B)$
where $B$ is bounded and $E=f(B)$.
We need the following elementary result on finite unions of rectifiable sets.  
We suspect this is well known, but we could not find a reference.

\begin{lemma}\label{lem:finiterect}
If $E$ is a finite union of $m$ rectifiable sets, then $E$ is $m$ rectifiable.
\end{lemma}
\begin{proof}
Let $E=\cup_{i=1}^nE_i$ where $E_i$ is the image of the bounded set $B_i\subset\real^m$
under $f_i$ which has Lipschitz constant $L_i$. Let $L=\max_{1\le i\le n}L_i$
and $M =\sup_{\bsy,\tilde\bsy \in E}\Vert \bsy-\tilde\bsy\Vert$.
We can choose points $\bst_i\in\real^m$ so that the 
translated sets $\tilde B_i = \{\bsx + \bst_i\in \real^m\mid \bsx\in B_i\}$
have $\Vert\bsx-\tilde\bsx\Vert \ge1$ 
whenever $\bsx$ and $\tilde\bsx$ are in different translated sets.
Let $B =\cup_{i=1}^n\tilde B_i$
and define the function $f$ on $B$ by $f(\bsx) = f_i(\bsx-\bst_i)$ whenever $\bsx\in \tilde B_i$.
Now $f$ is Lipschitz with constant at most $\max(L,M)$.
Then $B$ is bounded and $E$ is the image of $B$ under $f$.
\end{proof}

The conclusion to Lemma~\ref{lem:finiterect} does not hold for countable unions.
The next theorem is one of the few ways to get a conclusion about finite
Minkowski content.

\begin{theorem}\label{thm:rect2mink}
If  $W$ is a closed $m$ rectifiable subset of $\real^n$, 
then $\bm_m(W)=\ch_m(W)$. 
\end{theorem}
\begin{proof}
This is Theorem 3.2.39 of \cite{fede:1969}. 
\end{proof}

Federer's definition of $m$ rectifiability is stricter than we see in some
more recent works.  For example, $m$-rectifiability of Definition 15.3 of \cite{matt:1999}
is $(\ch_m,m)$ rectifiability of Definition~\ref{defn:rect}. That definition of rectifiability is easier for us to establish in our WoS problem but it is not strong enough to use in Theorem~\ref{thm:rect2mink}.

\subsection{The domain, its boundary and distance functions}

Our study of the WoS algorithm requires smoothness of
$z\mapsto \dist(\bsz,\partial\Omega)$
almost everywhere in $\Omega^\circ$. 
The required smoothness fails to hold at points $\bsz$ whose projection onto $\partial\Omega$ is not unique. We discuss such medial points below as well as some points that project
onto places where two smooth subcurves of $\partial\Omega$
meet. Under our conditions $\dist(\bsz,\partial\Omega)$ is analytic everywhere except on a finite union of analytic curves.  The results we need to study the smoothness of the function $\dist(\cdot,\partial\Omega)$ are scattered
over different publications with multiple definitions of the necessary concepts.

We use regularity conditions on $\Omega$ taken from \cite{choi:choi:moon:1997}.
First, $\Omega$ is the closure of a connected and bounded open
subset of $\real^2$.
The boundary $\partial\Omega$ is the union of finitely many
disjoint simple closed curves. Those are described as embeddings
of the unit circle into $\real^2$. 
Such a curve is closed, continuous and does not intersect itself.
It does not have to be diffeomorphic to the unit circle like the embeddings
from differential topology \citep{guil:poll:1974}. For example, it can have corners.
When there are $G+1$ of these curves, then the set $\Omega$
has one outer boundary curve and $G$ `holes' cut out of it by inner boundary curves. 
It is then said to have \emph{genus} $G$. The gasket example has $G=50$. Many problems have genus $0$.

Each of the $G+1$ boundary curves is the union of a finite number of 
pieces we call arcs. Each of those arcs is a real analytic curve defined on a closed interval $[a,b]$.
Those in turn are the restrictions of a real analytic curve over an open
interval such as $(a-\varepsilon,b+\varepsilon)$ to $[a,b]$.

\cite{choi:choi:moon:1997} have to exclude the possibility that
$\partial\Omega$ is actually a circle.  That case is easy
for our WoS theory and can also be sampled without WoS
by one dimensional integration of $h$ times a Poisson
kernel over $\partial\Omega$, so this exclusion does not cause
us difficulty.

\begin{definition}\label{def:ccmdomain}
A domain $\Omega$ satisfying the above stated conditions
of \cite{choi:choi:moon:1997} is called a CCM domain.
\end{definition}

Analyticity may seem like an extremely strong assumption, but
even $C^\infty$ curves can show pathologies.
See  Figures 1 and 2 of \cite{choi:choi:moon:1997} where sets with $C^\infty$ boundaries have medial axes consisting of a countably infinite number of curves with infinite total length.  
They also remark that many applications define boundaries
via non-uniform rational B-splines (NURBS) which satisfy their conditions,
along with more commonly considered line segments and circular arcs.

For Dirichlet boundary conditions and a harmonic function $u$,
we will examine the case where $u(\bsz)$ takes the value $0$ on all of $\partial\Omega$ except for one arc $\ca$. Such cases form a basis for more general problems.
We define the distances
\begin{align*}
r_{\partial\Omega}(\bsz) &= \dist(\bsz,\partial\Omega) = \min_{\tilde\bsz\in\partial\Omega}\Vert\bsz-\tilde\bsz\Vert,\quad\text{and}\\
r_\ca(\bsz) &= \dist(\bsz,\ca) = \min_{\tilde\bsz\in\ca}\Vert\bsz-\tilde\bsz\Vert.
\end{align*}
Both of these are minima not just infima because they are distances
to closed sets.  These functions are both $1$-Lipschitz
from the triangle inequality. 
\newline

Here is the definition of a geometric graph
from \cite{choi:choi:moon:1997}.

\begin{definition}\label{def:geograph}
  A set in $\real^2$ is a geometric graph if it is topologically a 
usual connected graph with a finite number of vertices and edges,
where a vertex is a point in $\real^2$
and an edge is a real analytic curve with finite length whose
limits of tangents at the end points exist.
\end{definition}

\begin{definition}\label{def:regularwos}
The WoS problem is regular if $r_{\partial\Omega}$ and $r_{\ca}$
are both real analytic functions on all of $\Omega^\circ$ except for a
subset $S$ of $\Omega^\circ$ 
that is a finite union of geometric graphs.
\end{definition}


Next, we consider smoothness of the distance functions.
For $k\ge2$, the distance to a $C^k$ curve $\mathcal{C}$ is $C^k$ 
over  $U\setminus\mathcal{C}$ where $U$ is a neighborhood
of $\mathcal{C}$ \citep{foot:1984}.  Our regularity condition requires 
greater smoothness and requires it over a larger
set than just locally near our curves.
We use the following definitions taken from Definition 1 of \cite{leob:stei:2021}
to give a result on global smoothness.

\begin{definition}\label{defn:medial}
Let $M\subseteq\real^d$ be nonempty.
The function $r_M:\real^d\to[0,\infty)$
has values $r_M(\bsx) = \dist(\bsx,M)$.
The set
$$\unpp(M) = \{\bsx\in\real^d\mid 
\text{there is a unique $\tilde\bsx\in M$ with $\Vert\bsx-\tilde\bsx\Vert=r_M(\bsx)$}\}$$
has the non-medial points of $M$.
The function  $p_M:\unpp(M)\to M$
is the projection $p_M(\bsx) = \argmin_{\tilde\bsx\in M}\Vert \tilde\bsx-\bsx\Vert$.
The set $\co(M)=(\unpp(M))^\circ$
is the largest open set contained in $\unpp(M)$.
\end{definition}

There can be points in $\unpp(M)\setminus\co(M)$.
These are non-medial points
to which some sequence of medial points converges.
Example 1.2 of \cite{dude:holl:1994} has
$M = (\real\times \{-1,1\})
\setminus \{(0,1)\}$. Then as $n\to\infty$, the medial points $(1/n,0)$ converge to the non-medial point $(0,0)$.

\begin{theorem}\label{thm:itssmooth}
Let $M$ be a $C^k$ submanifold of $\real^d$ with $k\ge1$. Then
the projection $p_M$ is $C^{k-1}$ on $\co(M)$ and 
the distance $r_M$ is $C^k$ on $\co(M)\setminus M$.
If $M$ is analytic then so are $p_M$ and $r_M$.
\end{theorem}
\begin{proof}
For $1\le k<\infty$, this is Theorem 2 of \cite{leob:stei:2021}.  
\cite{dude:holl:1994} have the results for $k\ge2$ and for the analytic case.
Their Theorem 4.1 covers $p_M$ and
their Corollary 4.5 covers $r_M$.
\end{proof}

\cite{choi:choi:moon:1997} define the medial axis of a CCM boundary
$\Omega$ in a way that is similar to $(\unpp(\partial\Omega))^c$
but with a few differences.
\begin{definition}\label{def:ma}
The medial axis of $\Omega$, denoted $\ma(\Omega)$
is the set of points $\bsz\in\Omega$ that are centers
of maximal inscribed disks.
That is $B(\bsz,r)\subset\Omega$
holds for some $r\ge0$ and if $B(\bsz,r)\subset B(\tilde\bsz,r')$ for $r'>r$
then  $B(\tilde\bsz,r')\not\subset\Omega$.
    \end{definition}

The medial axis $\ma(\Omega)$ in Definition~\ref{def:ma} differs from the set of points in Definition~\ref{defn:medial} with non-unique projections onto $\Omega$. Let $\wt\ma(\Omega)=(\unpp(\Omega))^c$ be that definition of medial points from \cite{leob:stei:2021}. Then $\wt\ma(\Omega)\setminus\ma(\Omega)$ is made up of points in the exterior of $\Omega$ with non-unique projections.  These do not affect the WoS algorithm running inside $\Omega^\circ$. The set $\ma(\Omega)\setminus \wt\ma(\Omega)$ contains corner points of $\partial\Omega$ where
a disk of radius $0$ is maximal but only touches $\partial\Omega$
at one point. For example, if $\Omega$ is a rectangle then $\ma(\Omega)$ has the diagonals including vertices while $\wt\ma(\Omega)$ excludes the vertices. Theorem 6.2 of \cite{choi:choi:moon:1997} shows that a CCM domain has only finitely many of these points.

The most important result from \cite{choi:choi:moon:1997} for our
present purposes is the following theorem.

\begin{theorem}\label{thm:fromccm}
Let $\Omega\subset\real^2$ be a CCM domain.
Then $\ma(\Omega)$ is a geometric graph.
\end{theorem}
\begin{proof}
This is Theorem 8.2 of \cite{choi:choi:moon:1997}.
\end{proof}

\begin{theorem}\label{thm:itswosregular} 
Let $\Omega$ be a CCM domain where all of the arcs in its definition are analytic. Then $r_{\partial\Omega}$ is analytic on $\Omega^\circ\setminus \ma(\Omega)$. 
Additionally, suppose that $\ca$ is one of those piece-wise analytic simple closed curves in $\partial\Omega$.
Then the WoS problem is regular.
\end{theorem}

\begin{proof} 
For the first claim, 
it is enough to show that $\Omega^\circ\setminus\ma(\Omega)\subset\unpp(\partial\Omega)$. 
To see why this is enough, we first note that the geometric graph $\ma(\Omega)$ is a closed set and hence $\Omega^\circ\setminus\ma(\Omega)$ is open.
Then once the above inclusion is established, it follows that $\Omega^\circ\setminus\ma(\Omega)$
is a subset of $\unpp(\partial\Omega)^\circ=\co(\partial\Omega)$. Also, since $\Omega^\circ\subset(\partial\Omega)^c$, we have  $\Omega^\circ\setminus\ma(\Omega)\subset\co(\partial\Omega)\setminus\partial\Omega$. 
Then the analytic case of Theorem~\ref{thm:itssmooth} applies to show that $r_{\partial\Omega}$
is analytic on $\Omega^\circ\setminus(\ma(\Omega))$.

Now suppose that $\Omega^\circ\setminus\ma(\Omega) \not\subset\unpp(\partial\Omega)$. 
Then there is a point $\bsz\in \Omega^\circ\setminus\ma(\Omega)$ with two or more nearest points in $\partial\Omega$. 
That point $\bsz$ would then be the center of a maximal disk inside $\Omega$ placing it in $\ma(\Omega)$ which is a 
contradiction, establishing that $r_{\partial\Omega}$ is analytic on $\Omega^\circ\setminus\ma(\Omega)$.

For $r_\ca$, choose a very large disk $D$ with $\Omega\subset D$
and $r_{\partial D}(\bsz) > r_{\ca}(\bsz)$ for all $\bsz\in\Omega$.
Now $\partial D$ and $\ca$ define the boundary of a CCM domain
$\Psi\supset \Omega$. By the above result, $r_{\partial\Psi}$
is an analytic function on $\Psi^\circ\setminus\ma(\Psi)$.
Then because $\Omega\subset \Psi$, $r_{\partial\Psi}$ is analytic
on $\Omega^\circ\setminus\ma(\Psi)$.  Because $D$ is so large,
$r_{\partial\Psi}(\bsz)=\min( r_{\partial D}(\bsz),r_\ca(\bsz))=r_\ca(\bsz)$ for all $\bsz\in\Omega$
and so $r_\ca$ is analytic on $\Omega^\circ\setminus\ma(\Psi)$.

Finally both medial axes $\ma(\Omega)$  and $\ma(\Psi)$
are geometric graphs. That makes the WoS problem regular.
\end{proof}

\begin{remark}
Theorem~\ref{thm:itswosregular} applies directly to the gasket example.  There $\Omega$ is the whole domain and $\ca$ could be the closed curve bounding any of the $50$ holes or it could be the exterior boundary curve.  In many applications of WoS, $\partial\Omega$ is made up of closed curves, that are in turn made up of finitely many pieces, and $u$ is constant within such pieces but not on the whole closed curve. That brings additional complexity which we study in Section~\ref{sec:splitcurve}.
\end{remark}

\subsection{Distance to a piece of $\partial\Omega$}\label{sec:splitcurve}

The distance to an arc $\ca$ that is just one analytic piece of one of the closed curves bounding $\Omega$ is much more complicated than the combination of \cite{dude:holl:1994}
and \cite{choi:choi:moon:1997}. Such a set $\ca$ has endpoints that
are not allowed by Theorem~\ref{thm:itssmooth} and it is not the boundary
of a compact set, which Theorem~\ref{thm:fromccm} requires. We did not see any treatment of this case in the literature that is as comprehensive as the study of CCM domains in \cite{choi:choi:moon:1997}. 

Here we study $r_\ca$ for some special cases and then sketch why we think its analyticity away from a geometric graph should hold more generally.

In the special cases we consider next, $\ca$ is either a line segment or a portion of a circular arc. We also assume that any straight line in $\real^2$ that intersects $\Omega$ does so in a finite set of disjoint closed intervals. Then analyticity holds for the distance from these curves, apart from points in a finite union of geometric graphs.
Many useful domains $\Omega$ can be constructed using line segments and circular
arcs as pieces of $\partial\Omega$. This includes all the two dimensional examples we consider in this paper. 

\subsection{$\ca$ is a line segment}
Suppose that $\ca$ is a finite line segment in $\real^2$. Without loss of generality,
$\ca = [0,1]\times\{0\}$.  Then 
$$r_\ca(\bsz) = \bigl(\dist(z_1,[0,1])^2+ z_2^2\bigr)^{1/2}
$$
which is analytic everywhere except $\ca\cup (\{0,1\}\times\real)$.
This exceptional set is the union of a line segment and two lines.  Its intersection with the convex hull of $\Omega$ is a geometric graph. The holes in $\Omega$ may split that geometric graph into a finite set of smaller geometric graphs but under our assumption they cannot yield an infinite set of geometric graphs.

\subsection{$\ca$ is part of circular arc}
Suppose instead that $\ca$ is a portion of a circular arc.
Without loss of generality $\ca = \{ \phi(t) \mid |t|\le c\}$
for $\phi(t) = (\cos(t),\sin(t))$ and $0<c<\pi$.
Write $\bsz$ in polar coordinates with radius $r(\bsz)\ge0$ and angle $\theta(\bsz)\in[-\pi,\pi)$.
Define regions
\begin{align*}
S_{\mathrm{In}} & =\{\bsz\in\real^2\mid |\theta(\bsz)|<c,\ 0<r(\bsz)<1\},\\
S_{\mathrm{Out}} & =\{\bsz\in\real^2\mid |\theta(\bsz)|<c,\ r(\bsz)>1\},\\
S_{\mathrm{Up}} &=\{\bsz\in\real^2\mid c<\theta(\bsz)<\pi\}, \quad\text{and}\\
S_{\mathrm{Down}}&=\{\bsz\in\real^2\mid -\pi <\theta(\bsz)<-c\},
\end{align*}
depicted in Figure~\ref{fig:circarc}.
A point $\bsz$ in $S_{\mathrm{In}}$ or $S_{\mathrm{Out}}$ projects onto 
$\phi(\theta(\bsz))$ and has $r_\ca(\bsz)=|r(\bsz)-1|$
which is analytic on those two sets.
If $\bsz\in S_{\mathrm{Up}}$
then $\bsz$ projects to $\phi(c)$ and $r_{\ca}(\bsz)=\Vert\bsz-\phi(c)\Vert$
is analytic on $S_{\mathrm{Up}}$. Similarly, $r_\ca$ is analytic on $S_{\mathrm{Down}}$.
The exceptional sets where $r_\ca$ is not analytic are $\ca$, 
the point $(0,0)$ which projects non-uniquely onto $\ca$,
the ray $R_0=\{(-t,0)\mid 0< t<\infty\}$ of points with $r_{\phi(c)}(\bsz)=r_{\phi(-c)}(\bsz)=\dist(\bsz,\ca)$,
and the two rays $R_a=\{r(\cos(a),\sin(a))\mid 0< r<\infty\}$ for $a=\pm c$
which each separate a region projecting to $\ca^\circ = \phi( (-c,c))$
from a region projecting to an endpoint of $\ca$.

As in the case of a line segment, $r_\ca$ is analytic on the convex hull of $\Omega$ apart from $\ca$ itself and three line segments through the origin. Then analyticity holds apart from a finite union of geometric graphs and the problem is WoS regular.

\begin{figure}
\centering
\includegraphics[width=.9\hsize]{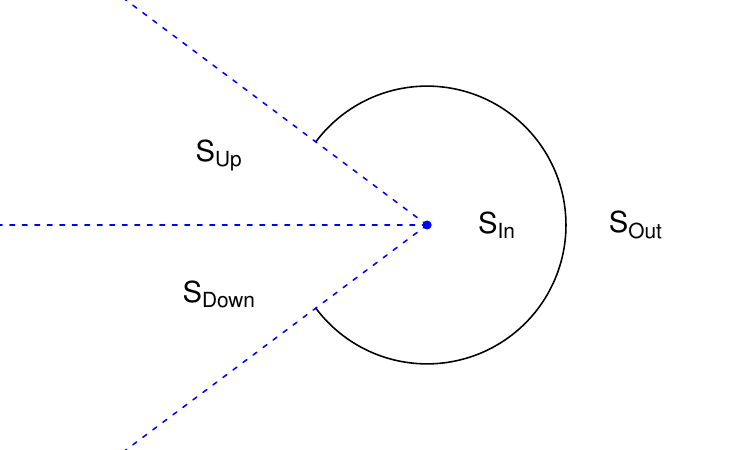}
\caption{\label{fig:circarc} The curve $\ca$ is a circular arc.
Together with the shown rays it defines $4$ open regions of $\real^2$ on which $r_\ca$ is analytic.
}
\end{figure}

\subsection{More general $\ca$}
The relevant literature for this case considers medial axes, point/curve bisectors and self bisectors of curves. Lacking results like those in \cite{choi:choi:moon:1997}, we sketch our reasons for believing that the distance to an analytic arc $\ca$ will be analytic apart from a geometric graph quite generally and not just for special cases like line segments and circular arcs.
We draw on a result in \cite{faro:rama:1998} which uses parameterized arcs. We write our real analytic arc as $\ca = \{\phi(t) \mid 0\le t\le 1\}$ and we assume that $\phi'(t)\ne\bszero$ at any $t\in[0,1]$.

With this parameterization, we define an internal curve $\ca^\circ = \phi( (0,1))$ and
endpoints $\bsa =\phi(0)$ and $\bsb=\phi(1)$.
Then
\begin{align}\label{eq:componentdistances}
r_\ca(\bsz)  = \min\bigl( r_{\bsa}(\bsz), r_{\ca^\circ}(\bsz),r_{\bsb}(\bsz)\bigr)
 = \min\bigl( \Vert \bsz-\bsa\Vert, r_{\ca^\circ}(\bsz),\Vert \bsz-\bsb\Vert\bigr).
\end{align}
The endpoint distances $\Vert\bsz-\bsa\Vert$ and $\Vert\bsz-\bsb\Vert$
are analytic at all but $\bsa$ and $\bsb$ respectively. The function $r_\ca$ can fail to be analytic at points where $r_{\ca^\circ}$ is not analytic or at points where two or more of the distances in~\eqref{eq:componentdistances} are equal.  We consider those equalities first and then describe how smoothness of $r_{\ca^\circ}$ is the gap that remains.

The set of points $\bsz$ where $\Vert\bsz-\bsa\Vert=\Vert\bsz-\bsb\Vert$ is a line $\ell_{\bsa\bsb}$ 
and so its intersection with $\Omega$ is contained within a finite union of line segments. Hence it is a finite union of geometric graphs.

We can never have $\dist(\bsz,\bsa)<\dist(\bsz,\ca^\circ)$.  Those distances are equal on the set 
$$
E_{\bsa} = \bigcap_{0<t<1}\bigl\{\bsz\mid \Vert\bsz-\bsa\Vert <\Vert \bsz-\phi(t)\Vert\bigr\}
$$
and there is an analogous set $E_{\bsb}$ at the other endpoint of $\ca$.
The set $E_{\bsa}$ is an intersection of open half-spaces so it is convex. For any $\bsz\in E_{\bsa}^{\circ}\setminus\ell_{\bsa\bsb}$, there is an open set $U$ with $\bsz\in U\subset E_{\bsa}\setminus\ell_{\bsa\bsb}$ on which $\dist(\bsz,\ca)=\dist(\bsz,\bsa)=\Vert\bsz-\bsa\Vert$ which is analytic. That argument does not work on $\partial E_{\bsa}$ which \cite{faro:rama:1998} discuss.

The set of points where $\Vert\bsz-\bsa\Vert =\dist(\bsz,\ca)$ is part of the point/curve bisector between $\ca$ and $\bsa$. Point/curve bisectors where the point is part of the curve are a degenerate case and Section 3.2 of \cite{faro:rama:1998} considers the bisector of a curve and one of its endpoints.  They say that the bisector is a subset of three curves, defined by their equations (a), (b) and (c). Curve (c) is the normal line  to $\ca$ at $\bsa$. Curve (b) is the linear end-tangent extension to $\ca$ at the endpoints. Curve (a) is defined using the unit normal curve $\nu(t)$ where $\nu(t)$ is $\phi'(t)/\Vert\phi'(t)\Vert$ rotated through 90 degrees.
Then curve (a) has  the expression
$$a(t)=
\phi(t) + \nu(t)\frac{|\bsa-\phi(t)|^2}{2(\bsa-\phi(t))^\tran\nu(t)}
$$
and we interpret $a(0)$ as $\bsa$.   Now if $\phi$ and $\nu$ are both analytic on $\ca^\circ$, then so is~$a$ on intervals where the denominator does not vanish.  As $t$ approaches an interior point where the denominator vanishes, $\Vert a(t)\Vert$ diverges to infinity, thereby leaving the domain $\Omega$. As a result we expect the curve $a(t)$ to intersect $\Omega$ only in a finite union of geometric graphs.

The gap is that we did not find conditions in the literature to ensure that  $r_{\ca^\circ}$ is analytic apart from a finite union of geometric graphs. We would need the complement of $(\unpp(\ca^\circ))^\circ$ to be a finite union of geometric graphs. 
Theorem 3.9 of \cite{basu:pras:2023} is similar to what we need but we could not determine whether it applies to $\ca^\circ$ nor whether the resulting one dimensional simplicial complex it describes would qualify as a finite union of geometric graphs.

\section{Differential geometry theorems}\label{sec:diffgeo}

We need some results from differential geometry.  These are not commonly used in the RQMC literature so we state them here. We will not need the most general versions of these theorems, just the ones with statements that match our uses.

We need the notions of regular and critical values of a function as well as regular and critical points. For a smooth function $f:M\to N$ between manifolds the
point $\bsx\in M$ is a critical point of $f$ if the Jacobian of $f$ at $\bsx$ has less than full rank.  
Then $\bsy = f(\bsx)$ is a critical value of $f$.  
If $\bsx$ is not a critical point of $f$ then
it is a regular point of $f$.
If every $\bsx$ with $f(\bsx)=\bsy$ is a regular point
of $f$ then $\bsy$ is a regular value of $f$.
A regular level set is the level set of a regular value.

\begin{theorem}[Pre-image theorem]\label{thm:preimage}
Every regular level set of a smooth
map between smooth manifolds is a properly embedded submanifold whose codimension is equal to the dimension of the codomain.
\end{theorem}
\begin{proof}
This is the statement from Corollary 5.14 of \cite{lee:2013}.
\end{proof}
Smooth means $C^\infty$ throughout \cite{lee:2013} and this is enough for our purposes.
The level set of $f:M\to N$ has dimension equal to the dimension of $M$ minus
the dimension of $N$, which in our case will be the dimension of $M$ minus~$1$.
\cite{guil:poll:1974} call their version of this result,  the pre-image theorem.  Because the level set is an embedded manifold, it is smooth \citep[page 98]{lee:2013}. 
We will not reference the properness of this embedding.

\begin{theorem}[Rademacher's theorem]\label{thm:rademacher}
Let $f:\real^n\to\real^m$ be locally Lipschitz.  Then $f$ is differentiable almost everywhere.
\end{theorem}
\begin{proof}
This version of Rademacher's theorem is given as Theorem 3.2 of \cite{evan:gari:2015}.
  \end{proof}

  \begin{theorem}[The coarea formula]\label{thm:coarea}
    Let $f:\real^n\to\real$ be Lipschitz continuous.
    Then for each Lebesgue measurable set $A\subseteq\real^n$,
    $$
\int_A \Vert\nabla f(\bsx)\Vert\rd\bsx = \int_{\real}\ch_{n-1}(A\cap f^{-1}(y))\rd y.
    $$
    \end{theorem}
    \begin{proof}
      This version of the coarea formula is Theorem 3.10 of \cite{evan:gari:2015}
      specialized to a real valued function $f$.
    \end{proof}

\begin{theorem}[Sard's theorem]\label{thm:sards}
  Suppose that $M$ and $N$ are smooth manifolds with or without
  boundary and $f:M\to N$ is a smooth map. Then the set of critical
  values of $f$ has measure zero in $N$.
\end{theorem}
\begin{proof}
This version of Sard's theorem is from Theorem 6.10 of \cite{lee:2013}.
\end{proof}

\begin{theorem}[Implicit function theorem]\label{thm:implicit}
  Let $U\subseteq \real^m\times \real$ be an open subset,
  and let $(\bsw,y)$ for $\bsw\in\real^m$ and $y\in\real$ denote
  the standard coordinates on $U$. Suppose $\Phi:U\to\real$ is a smooth
  function, $\bsa,b\in U$ and $c=\Phi(\bsa,b)$. If 
  $\frac{\partial\Phi}{\partial y}(\bsa,b)\ne0$, then there exist neighborhoods
  $V_0\subseteq \real^m$ of $\bsa$ and $W_0\subseteq \real$ of $b$ and
  a smooth function $F:V_0\to W_0$ such that $\Phi^{-1}(c)\cap(V_0\times W_0)$
  is the graph of $F$, that is, $\Phi(\bsw,y)=c$ for $(\bsw,y)\in V_0\times W_0$ if and
  only if $y=F(\bsw)$.
\end{theorem}
\begin{proof}
  This version of the implicit function theorem is Theorem C.40
  of \cite{lee:2013} specialized to the case where $y$ is one dimensional.
\end{proof}

\section{Minkowski content for WoS}\label{sec:main}

For $d=2$ our RQMC algorithm constructs $\bsz_k$ from $\bsz_0$ and a point $\bsx\in[0,1)^k$.
Here we consider the set of input points in $[0,1)^k$ that cause $\dist(\bsz_k,\partial\Omega)<\varepsilon$ along with $u(\bar\bsz_k)=1$.
We give conditions under which that set has 
$\mupkmo<\infty$ for almost all $\varepsilon$.
A separate analysis in Section~\ref{sec:rqmcforwos} accounts for the possibility that the walk may have terminated at $k'<k$ steps.

We assume that $\Omega\subset \real^2$ is a CCM domain.
Then $\partial\Omega$ is a finite union of closed analytic arcs one of which is $\ca$.
The target function $u:\partial\Omega\to\{0,1\}$ takes the value $1$ on 
$\ca$  and is $0$ in $\partial\Omega\setminus\ca$.
We abbreviate $r_{\ca}$ to $r_1$, $r_{\partial\Omega\setminus\ca}$ to $r_0$
and $r_{\partial\Omega}$ to $r$.

Now let
\begin{align}\label{eq:defcp}
\cp = \bigl\{\bsz\in\Omega^\circ\mid \text{$r_1$ is not analytic}\bigr\}
\bigcup\, \bigl\{\bsz\in\Omega^\circ\mid \text{$r$ is not analytic}\bigr\}
\end{align}
be the collection of points in $\Omega^\circ$ where at least one of
our distance functions fails to be analytic.
For a WoS regular problem, $\cp$ is a finite union of geometric graphs.

The WoS problems we consider start at $\bsz_0\in\Omega$ with $r(\bsz_0)>\varepsilon$.
A point $\bsx\in[0,1)^k$  generates steps $\bsz_j(\bsx)$ for $j=1,\dots,k$ via
$$
\bsz_{j}(\bsx) = \bsz_{j-1}(\bsx) + r(\bsz_{j-1}(\bsx))
\theta(x_j),\quad\text{for $\theta(x)= (\cos(2\pi x),\sin(2\pi x))^\tran$}.
$$
By convention, $\bsz_0(\bsx)=\bsz_0$ a constant function on $[0,1)^k$.
The function $\theta$ above is Lipschitz continuous with
a constant of $2\pi$ equal to the norm of its derivative.

We will use the set
$$\cek =\bigcup_{j=1}^k \bsz_j^{-1}(\cp)\subset\toru^k$$
to represent all walks that ever hit a problematic point
in their first $k$ steps.
Under RQMC sampling, $\Pr(\bsx\in\cek)=0$. 

We will use the following `chamber regularity' assumption
on $\cek$. This assumption is illustrated by Figure~\ref{fig:twosurfaces} with a context described in the proof of Theorem~\ref{thm:wosok}.
\begin{definition}\label{defn:chamber}
The set $\cek$ has chamber regularity if it is contained in the union
of a finite number of smooth $k-1$ submanifolds of $\toru^k$
and for any point $\bsp\in\cek$
there exists a neighborhood $U$ of $\bsp$ 
for which $U\setminus\cek$ has finitely many connected components.
We call those the chambers of $U\setminus\cek$.
\end{definition}
This definition does not allow for two of the manifolds in $\cek$
to intersect infinitely often within any neighborhood of $\bsp$.
That conclusion might possibly follow from arguments using analyticity
similar to those used in \cite{choi:choi:moon:1997} and we think
that WoS regularity is almost enough to give chamber
regularity, but exploring these points is outside the scope of this article.

It is very convenient that $\bsz_k$ is a periodic function on $\real^k$.
We can therefore choose the domain to be the flat torus $\toru^k=\real^k/\ints^k$.
This is a smooth manifold without boundary. 
We will need to apply the implicit function theorem to a function
defined on a flat torus.  That theorem is stated for functions on Euclidean space.
It concerns local properties of the function.  A small neighborhood $U$
of $\bsx\in\toru^k$ can be identified with a small neighborhood $\tilde U$
around a representative of $\bsx$ in Euclidean space.
We can then view a function on $U\subset\toru^k$
as a function on $\tilde U\subset [-1,2)^k\subset\real^k$.

The $k$-step distance to $\ca$ is
$$
r_{k,1}(\bsx) = \dist(\bsz_k(\bsx),\ca) =r_1(\bsz_k(\bsx)).
$$
For RQMC we want  $\mupkmo(\{\bsx\in[0,1)^k\mid r_{k,1}(\bsx)=\varepsilon\})<\infty$.

\begin{theorem}\label{thm:wosok}
For fixed $k\ge1$, a WoS regular problem where $\cek$ 
satisfies chamber regularity
has $\bm_{k-1}(r_{k,1}^{-1}(t))<\infty$   for almost all $t>0$.
\end{theorem}
\begin{proof}
For $t>0$, define the level set $\cl(t) = r_{k,1}^{-1}(t)$.
The function $r_{k,1}$ is  Lipschitz because it is
a composition of Lipschitz functions.

The function $r_{k,1}$ is real analytic on $\toru^k\setminus \cek$
by induction starting with the constant function $\bsz_0$ and using
Dudek and Holly's Theorem~\ref{thm:itssmooth} and analyticity of $\theta$
at each step of an induction argument for $\bsz_j = \bsz_{j-1}+r(\bsz_{j-1})\theta(x_j)$.
Because $r_{k,1}$ is Lipschitz, it has a gradient almost everywhere
on $\toru^k$ by Rademacher's Theorem (Theorem~\ref{thm:rademacher}).
Then the coarea formula 
(Theorem~\ref{thm:coarea}) gives
\begin{align}\label{eq:coarea}
\int_{[0,\infty)}\ch_{k-1}(r_{k,1}^{-1}(t))\rd t
=\int_{(0,1)^k} \Vert \nabla r_{k,1}(\bsx)\Vert\rd\bsx \le L_{k,1}<\infty
\end{align}
where $\ch_{k-1}$ is the $k-1$ dimensional Hausdorff measure,
$L_{k,1}$ is the Lipschitz constant for $r_{k,1}$.

It follows that almost all level sets of $r_{k,1}$ have finite $k-1$ dimensional
Hausdorff measure.  That is not yet enough to give  a finite $k-1$ dimensional
Minkowski content.  We will use
Theorem~\ref{thm:rect2mink} from \cite{fede:1969}.
For every point $\bsp\in\cl(t)$ we will find an open set $U_{\bsp}\in\toru^k$
on which $\cl(t)$ is finitely $k-1$ rectifiable.
These sets cover $\cl(t)$ which is compact and so they have a finite
subcover.  Then Lemma~\ref{lem:finiterect} shows that $\cl(t)$ is $k-1$
rectifiable which is what we need to apply Theorem~\ref{thm:rect2mink}.
There are two cases to consider.  Either $\bsp\in \cl_0(t) = \cl(t)\cap\cek^c$
or $\bsp\in \cl_1(t) =\cl(t)\cap\cek$.

We consider $\bsp\in\cl_0(t)$ first.
The function $r_{k,1}$ is real analytic on $\cl_0(t)$ and so
we may apply Sard's theorem  (Theorem~\ref{thm:sards})
to show that almost every value of $t$ is
a regular value of $r_{k,1}$ restricted to an open neighborhood of  $\toru^k\setminus \cek$. For such $t$, $\nabla r_{k,1}$ is nonzero on $\cl_0(t)$
and then by the pre-image theorem (Theorem~\ref{thm:preimage})
$\cl_0(t)$
is a $C^\infty$ submanifold of dimension $k-1$.  The smoothness level has been
reduced from analytic to $C^\infty$ because
our version of the pre-image is just for $C^\infty$.  This will be enough for us.

For a regular value $t$ of $r_{k,1}$, fix  $\bsp\in\cl_0(t)$.
Then we can  apply the implicit function theorem (Theorem~\ref{thm:implicit})
to a real-valued function $\Phi$ defined on a neighborhood of $\bsp$ 
with $\Phi(\bsx) = r_{k,1}(\bsx)-t$.
Recall that the implicit function theorem can be used for functions on $\toru^{k-1}$.
Because $t$ is a regular value of $r_{k,1}$, at least one component of $\nabla \Phi(\bsp)$
is not zero. Let that component be $j\in\{1,2,\dots,k\}$.
Now for $\bsp_{-j}=(p_1,\dots,p_{j-1},p_{j+1},\dots,p_k)\in\toru^{k-1}$ there
are neighborhoods $V$ of $\bsp_{-j}$ and $W$ of $p_j$
and a $C^\infty$ function $F:V\to W$ on which 
$$\Phi^{-1}(0)\cap(V\times W)= \{(\bsx_{-j}, F(\bsx_{-j}))\mid \bsx_{-j}\in V\}.$$
That is $\cl_0(t)\cap (V\times W)$ is locally the graph of a $C^\infty$ 
function over an open subset of $\real^{k-1}$.
Then on a possibly smaller neighborhood around $\bsp\in\cl_0(t)$, $F$ is Lipschitz.
Then $\cl(t)$ is rectifiable on that neighborhood.

The case of $\bsp\in \cl_1(t)$ is more complicated.
Now $\bsp$ belongs to a smooth $k-1$ manifold where
$r_{k,1}$ fails to be smooth but is still Lipschitz.
Here we use our chamber regularity condition.
Suppose that we can pick a small neighborhood $U_{\bsp}$ of $\bsp$
that intersects only one of the smooth submanifolds in $\cek$, call it $\ce$.
That $\ce$ partitions $U_{\bsp}$ into three parts, $U_{\bsp,+}$, $U_{\bsp,-}$ and
$U_{\bsp,0}=U_{\bsp}\cap \ce$.  The chambers $U_{\bsp,+}$ and $U_{\bsp,-}$ 
are on `opposite sides' of $\ce$.
See Figure~\ref{fig:twosurfaces}.

More generally, there may be $s\ge1$ submanifolds of $\cek$
that intersect $U_{\bsp}$.  Then within $U_{\bsp}$ there are a finite
number of disjoint $k$ dimensional open chambers separated by the
submanifolds, generalizing $U_{\bsp,\pm}$.
Each of those chambers intersects $\cl(t)$ in a $k-1$ rectifiable
set by the same implicit function theorem argument we used in $\cl_0(t)$.

The remaining points of $\cl(t)\cap U_{\bsp}$ lie within the chamber boundaries. They are in the union of $s$  $k-1$ dimensional sets 
of the form $U_{\bsp}\cap \ce$
where $\ce$ is one of the submanifolds of $\cek$.  
Next, we show that $\cl(t)\cap U_{\bsp}\cap \ce$ is $k-1$ rectifiable.
Because $\ce$ is smooth and $U_{\bsp}$ is open $U_{\bsp}\cap\ce$
is a smooth $k-1$ manifold and hence is $k-1$ rectifiable.  Then
$\cl(t)\cap U_{\bsp}\cap\ce$, as a subset of $U_{\bsp}\cap\ce$,
is also $k-1$ rectifiable.  
Now $\cl(t)\cap U_{\bsp}$ is
the union of finitely many $k-1$ rectifiable chambers and
$s$  $k-1$ rectifiable sets from within the manifolds $\ce$.
That makes it $k-1$ rectifiable by Lemma~\ref{lem:finiterect}.

To complete the proof cover each point $\bsp\in\cl(t)$ by an open set $U_{\bsp}$ on which $\cl(t)\cap U_{\bsp}$ is $k-1$ rectifiable.
Then $\cl(t)$ is $k-1$ rectifiable by applying Lemma~\ref{lem:finiterect}
over the finite subcover mentioned above. Then $\bm_{k-1}(\cl(t))<\infty$ for almost all~$t$.
\end{proof}

\begin{figure}[t!]
\centering
\vspace*{-1.5cm}
\includegraphics[width=.9\hsize]{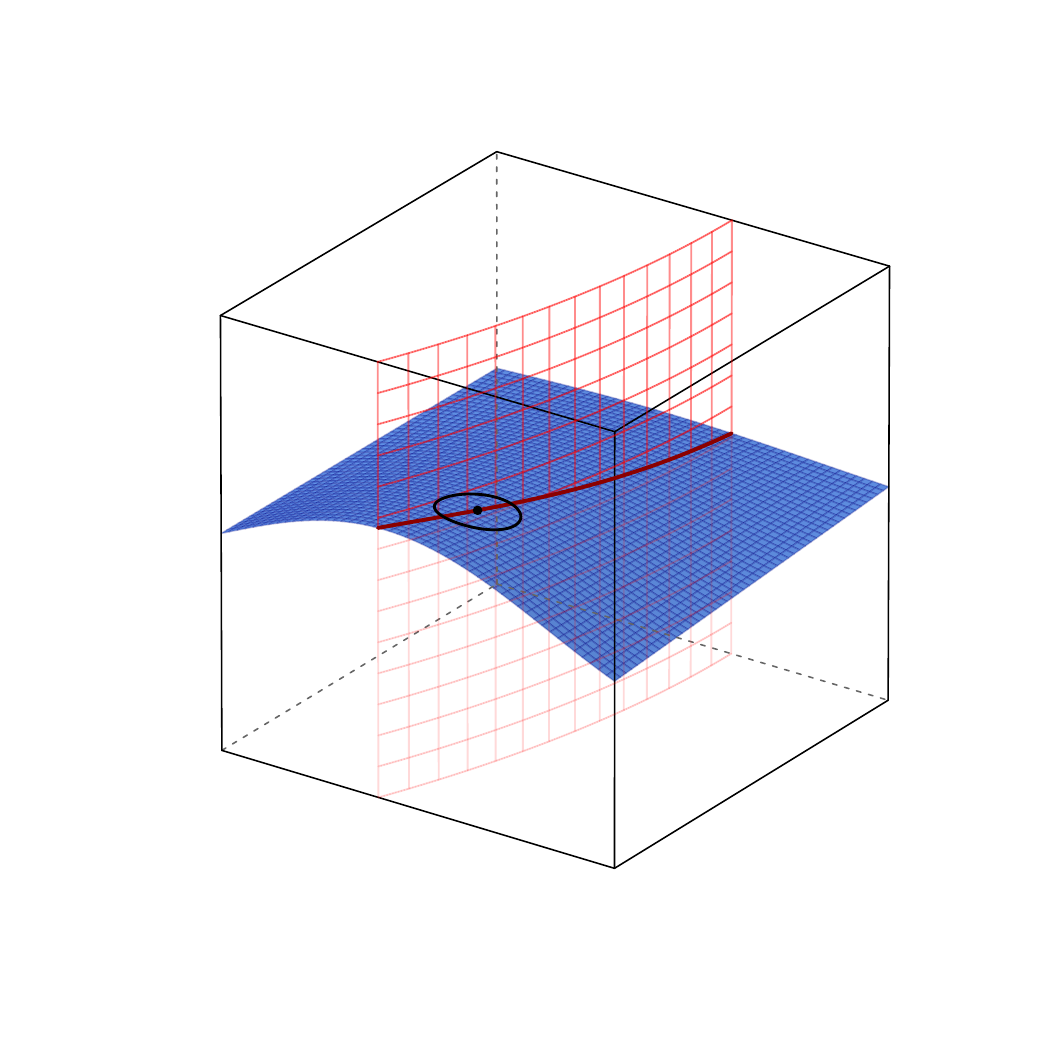}
\vspace*{-1.5cm}
\caption{\label{fig:twosurfaces} 
This figure illustrates a portion of the proof of Theorem~\ref{thm:wosok}
for $k=3$ (using nonperiodic functions).  The horizontal manifold (blue)
is the level set $\cl(t)$.  The vertical manifold (red) is the exceptional
set $\ce$ where the distance function is not smooth.  The intersection
of those manifolds is shown as a dark curve (red). The black point
is $\bsp\in\cl(t)\cap\ce$. The black curve is $\cl(t)\cap\partial U_{\bsp}$.
The neighborhood $U_{\bsp}$ has three parts, a middle part $U_{\bsp}\cap\ce$
and two parts on opposite sides of $\ce$.
}
\end{figure}

\section{Randomized quasi-Monte Carlo for WoS}\label{sec:rqmcforwos}

Here we use the prior results on Minkowski content to
understand how RQMC works on WoS problems.
For integers $k\ge1$ we can use RQMC points to estimate
\begin{align*}
\eta_{k,\ell} &= 
\int_{[0,1)^k} 1\{ r_{k,\ell}(\bsx) < \varepsilon\}\rd\bsx\quad\text{by}\quad 
\hat\eta_{k,\ell} = 
\frac1n\sum_{i=1}^n 1\{ r_{k,\ell}(\bsx_i) < \varepsilon\}
\end{align*}
for $\ell=0,1$.
By the finite $(k-1)$-dimensional Minkowski content of the level sets
established in the previous subsection, we can apply
Theorem 4.4 of \cite{he:wang:2015}
to show for almost all~$\varepsilon$ that
$\var( \hat\eta_{k,\ell} ) = O( n^{-1-1/k})$.
We may not know in practice whether our $\varepsilon$ value is a favorable
one.  By the same token, it would be an odd coincidence for the geometry
of $\ca$ and $\partial\Omega$ to make $\varepsilon$ an unfavorable value.
Similarly, the results from \cite{liu:2025} support this rate for $h(\bar{\bsz}_k(\bsx))1\{r_{k,\ell}(\bsx)<\varepsilon\}$ when $h(\bar{\bsz}_k(\bsx))$ is sufficiently smooth.

Next, we estimate
$$
\mu_{k,1}
=\int_{[0,1)^k} 1\{ r_{k,1}(\bsx) < \varepsilon\}
\times\prod_{j=1}^{k-1} 1\{ r_{j,1}(\bsx) \ge \varepsilon\}
\times\prod_{j=1}^{k-1} 1\{ r_{j,0}(\bsx) \ge \varepsilon\}
\rd\bsx
$$
by
$$
\hat\mu_{k,1} = \frac1n\sum_{i=1}^n
 1\{ r_{k,1}(\bsx_i) < \varepsilon\}
\times\prod_{j=1}^{k-1} 1\{ r_{j,1}(\bsx_i) \ge \varepsilon\}
\times\prod_{j=1}^{k-1} 1\{ r_{j,0}(\bsx_i) \ge \varepsilon\}.
$$
The integrand in $\mu_{k,1}$ is the indicator of the intersection
of $2k-1$ sets. The boundary of the intersection is contained
in the union of those $2k-1$ boundaries each of which has $k-1$ dimensional
Minkowski content for almost all $\varepsilon$.  As a result, this set either has $k-1$ dimensional
Minkowski content or has upper $k-1$ dimensional Minkowski
content of zero and so we get the same rate. 

Our estimate of $\mu_1 = \e( u(\bsz_0))$ is then
$\hat\mu_1=\sum_{k=1}^K \hat\mu_{k,1}$
for large $K$. This large value of $K$ introduces
a small bias in addition to the $O(\varepsilon)$ bias that comes from
using $\varepsilon>0$.
The $k$-th term in $\hat\mu_1$ has variance $O(n^{-1-1/k})$. An empirical rate like $n^{-1.1}$ can be interpreted as the integrand being effectively $10$ dimensional.  This doesn't mean that $\hat\mu_{10,1}$ has the most variance.  The variance of the sum includes faster and slower terms. We expect that the variance includes a contribution larger than $cn^{-1-1/K}$ for some $c>0$ whenever $n\ge N$ for a possibly very large $N$. We saw no sign of that slowing in our examples with $n\le 2^{17}$.


\section{Further examples and algorithms}\label{sec:moreexamples}

Here we present some further numerical examples
to explore different aspects of the RQMC for WoS problem. We consider examples with a known solution that lets us investigate bias, and we vary the gasket example to make it less favorable to RQMC. We also explore examples not covered by our theory, having nonzero source terms or having $\Omega\subset\real^3$. What we see in those examples is that our sufficient conditions do not appear to be necessary. We also considered some alternative samplers designed to mitigate the non-smoothness of the WoS problem.

\subsection{Unit disk example}
\cite{masc:hwan:2003} consider the Dirichlet BVP where $\Omega=B_2(\bszero,1)$ with $\Delta u(\bsz)=0$ for $\bsz\in\Omega$ and
\begin{align}\label{eq:mhdisk}u(\bsz)=\frac{1}{2}\ln[(z_1-2)^2+z_2^2],\textrm{ for } z\in \partial\Omega.
\end{align}
Because $\Delta u(\bsz)=0$ for $\bsz\ne(2,0)$ the formula in~\eqref{eq:mhdisk} also works for $\bsz\in\Omega$. We can use that to estimate bias and mean squared error (MSE).

Starting from $\bsz_0=(0,0.5)$, we run the RQMC-WoS algorithm, terminating once the walk hits the $\varepsilon$-shell of $\Omega$, for $\varepsilon=10^{-4}$. We run $100$ replicates for each sample size.
The results are shown in Figure~\ref{fig:naive1}. The bias in this example is negligible as expected and so the variance is essentially the MSE.  

Figure 2 of \cite{masc:kara:hwan:2004} shows results from a QMC solution on $n=20$ trajectories attaining an error of roughly $10^{-8}$.  They do not specify the QMC points in use or the starting point $\bsz_0$, and their Figure 3 shows essentially no difference between MC and QMC solutions for $n$ ranging to about $10^5$ in another test problem where they do give $\bsz_0$ and name the sequence. As a result, we do not think that the reported accuracy in their Figure~2 is directly comparable to the errors shown in our Figure~\ref{fig:naive1}.

\begin{figure}
    \centering
    \includegraphics[width=0.8\linewidth]{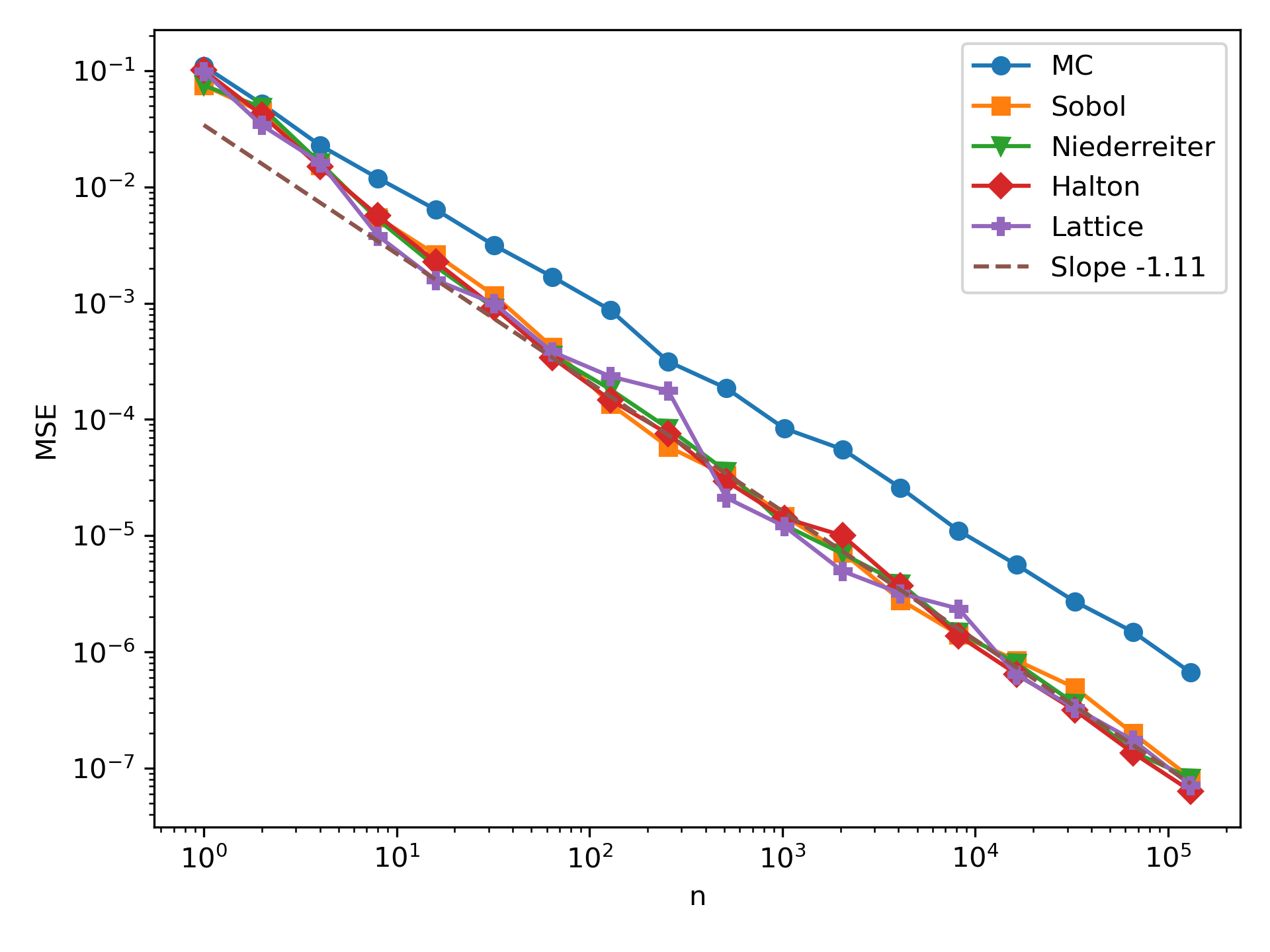}
    \caption{MSE of the WoS estimator for the unit disk example, starting at $\bsz_0 = (0,0.5)$ and using $\varepsilon=10^{-4}$.}
    \label{fig:naive1}
\end{figure}

\subsection{Gasket revisited}

We tried another starting point, midway between the two central bore holes, at $\bsz_0=(0.0030105,0.002839)$. This point is of interest because it is very close to two of the hottest parts of $\partial\Omega$.  What we found was that the RQMC algorithms gave essentially the same accuracy as MC did.
Most walks from this starting point hit one of the two nearby bore holes. Since these two bore holes have the same boundary temperature, deviations from that common value mainly come from walks that reach more distant boundary components, which typically require larger values of $k$. Since RQMC is less favorable for larger $k$, this can weaken the observed variance rate. To test this explanation, we lowered the temperature of the left bore hole to $140^\circ$C while leaving the right bore hole temperature at $160^\circ$C. This increased the importance of the walks hitting the nearby bore holes, and then we saw an RQMC variance rate of about $n^{-1.1}$.

We also investigated the bias in our first gasket example.  We averaged estimates from $100$ MC replicates with $n=2^{18}$ trajectories each and used that as a reference value.  We found very small bias and the variance and MSE were quite close.

\subsection{Major circular sector example}
We consider another two-dimensional example from \cite{masc:hwan:2003}. In polar coordinates, $\Omega=\{(r,\theta)\mid 0\le r\le1, -3\pi/2\le\theta\le0\}$, resembling the 1980s video game character Pac-Man rotated through $\pi/4$ radians.

This example has a nonzero source term, $\Delta u(r,\theta)=-(2-r^2)e^{-r^2/2}$ and so we use the WoS update from Section~\ref{subsec:source}. That update takes $3$ uniform variables per step, one for the angle to sample on the circle and two to sample in the disk.
Our theory does not cover this case because we have not studied the RQMC error for an integrand that evaluates $\Delta u$ uniformly over the WoS disk at step $k$, so it is of interest to see how RQMC performs.

The boundary conditions are
$u(r,0)=e^{-r^2/2}$, $u(r,-3\pi/2)=-r^{1/3}+e^{-r^2/2}$ and $u(1,\theta)= \sin(\theta/3)+e^{-1/2}$. 
The analytic solution is 
$u(r,\theta)=r^{1/3}\sin(\theta/3)+e^{-r^2/2}.$

Using RQMC points in this WoS problem gives the results shown in Figure \ref{fig:ex2_naive_mse}. We again see a slightly better rate of convergence for the MSE of the RQMC-WoS estimator. This example has an anomaly that we have not seen in any other: RQMC with the Niederreiter points performs worse than plain MC sampling does until about $n=2^{14}$ where it has nearly equal performance. It shows a rate comparable to the other RQMC methods but has a larger constant factor.

\begin{figure}[t]
    \centering
    \includegraphics[width=0.8\linewidth]{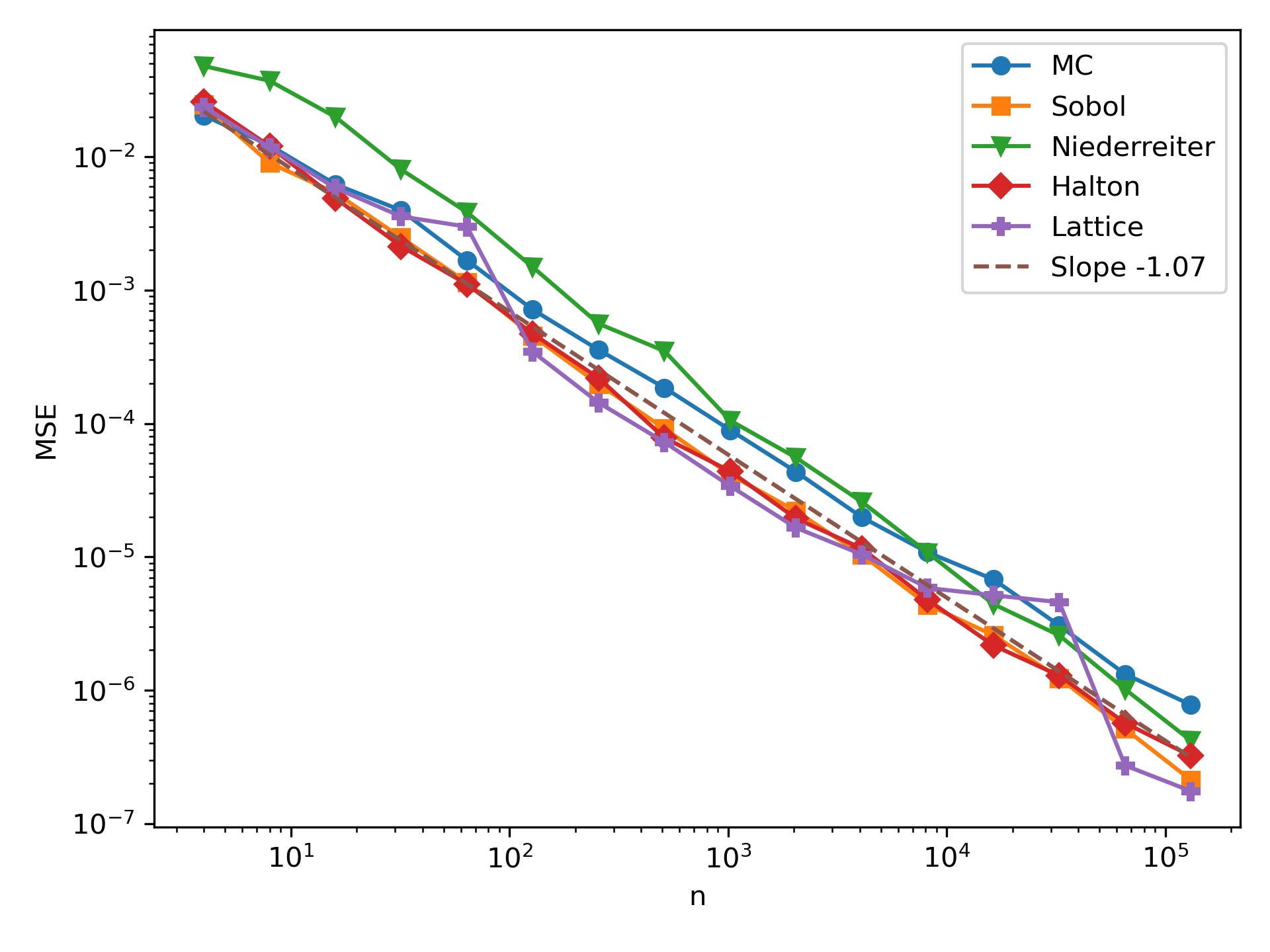}
    \caption{MSE of the standard WoS estimator for the major circular sector domain, for $\bsz_0=(r_0 \cos(\theta_0),r_0 \sin(\theta_0))$ with $(r_0,\theta_0)=(0.1244,-0.7906)$, and $\varepsilon=10^{-4}$.}
    \label{fig:ex2_naive_mse}
\end{figure}

\subsection{Dumbbell example}\label{sec:dumbbell}

This example is inspired by \cite{lund:rama:2019} and \cite{magn:pogg:2022}. They respectively study the number of critical points of $u$ and their distance to the boundary of their domain $\Omega$. Their $u$ is the solution to the BVP with
\begin{equation}\label{eq:dumbbell}
\text{$\Delta u(\bsz) = -2$,\,\ for $\bsz \in \Omega$}, \quad\text{and}\quad
u(\bsz) = 0, \,\ \text{for $\bsz \in \partial\Omega$}.
\end{equation}
Their $\Omega=B_2((-L,0),R)\cup([-L,L]\times [-w,w])\cup B_2((L,0),R)$
is a dumbbell-shaped region, the union of a $2L\times 2w$ bar with tiny $w>0$ and two circles of radius $R<L$. See Figure \ref{fig:dumbbell}. We will use RQMC-WoS to evaluate $u$ at a point $\bsz_0=(L-R,0)$ marked in the figure. Because this problem includes a source term it is not covered by our theorems, and instead provides insight into whether the conditions there are necessary. The interpretation in \cite{magn:pogg:2022} is that $u(\bsz)$ is the flow velocity of a fluid through a pipe whose cross-sectional shape is $\Omega$.

\begin{figure}
    \centering
    \includegraphics[width=1\linewidth]{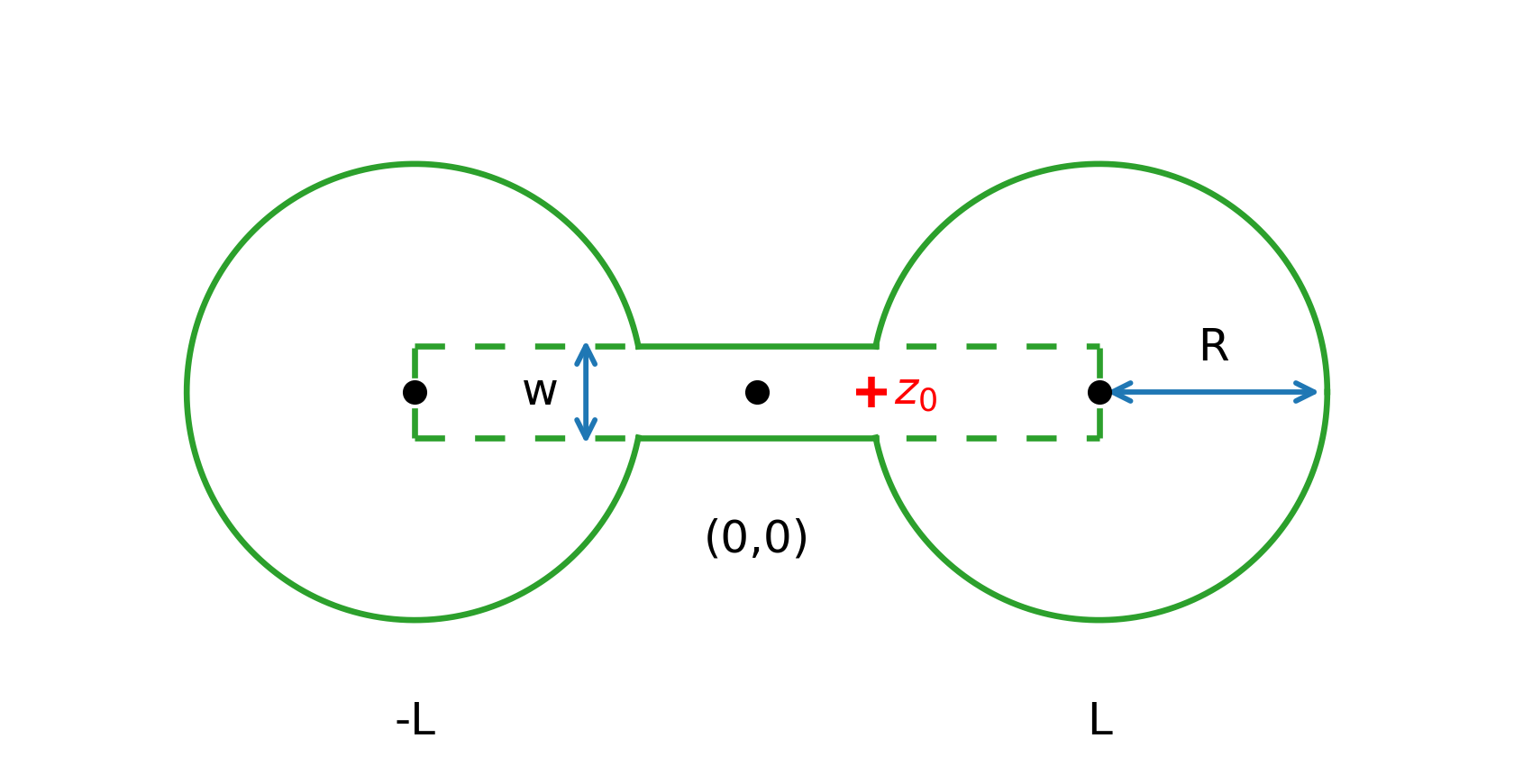}
    \caption{Dumbbell-shaped domain with parameters $R=1$, $L=1.5$ and $w=0.4$. The starting point $\bsz_0=(0.5,0)$ used in our experiments is marked with a red cross.}
    \label{fig:dumbbell}
\end{figure}

We can use equation~\eqref{eq:wosconstsource} on this problem. 
Then an MC or RQMC based estimate of $u(\bsz_0)$ takes the form
$$\hat{u}(\bsz_0)=\frac{1}{2n}\sum_{i=1}^n\sum_{k=1}^{\tau_{i}}
\dist(\bsz_{i,k-1},\partial\Omega)^2
$$
where $\tau_i = \tau_{i,\varepsilon}=\min\{k\in\natu_0\mid \dist(\bsz_{i,k},\partial\Omega)<\varepsilon\}$.

Using MC and RQMC samples, we get the results presented in Figure \ref{fig:dumbbell_var}. Again, we observe that using RQMC points brings a better rate of convergence for the variance of the WoS estimator.

\begin{figure}
    \centering
    \includegraphics[width=0.8\linewidth]{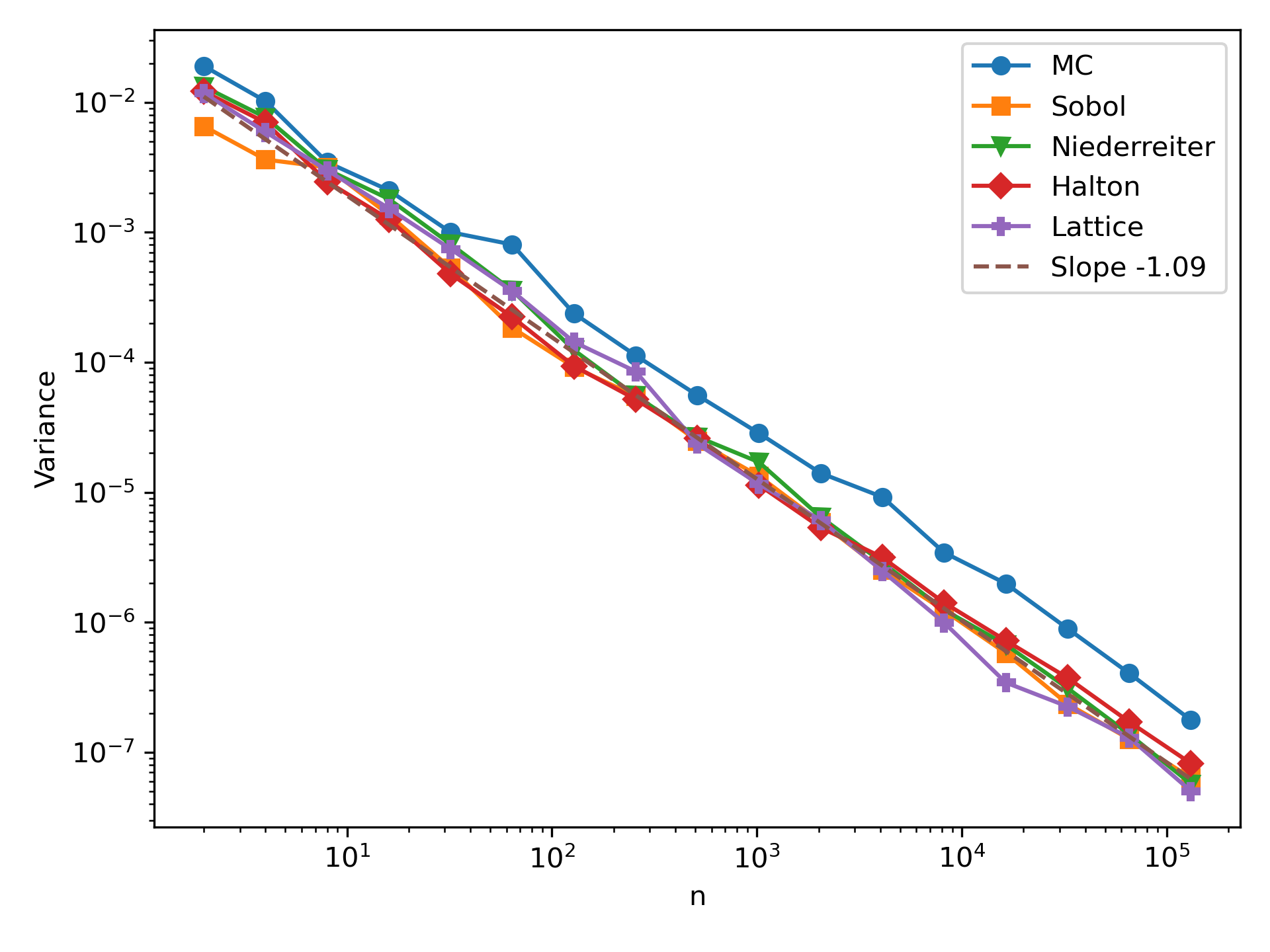}
    \caption{Variance of the WoS estimator for the dumbbell example, started at $z_0=(0.5,0)$ on the dumbbell domain, with $R=1$, $L=1.5$, $w=0.4$ and $\varepsilon=10^{-4}$.}
    \label{fig:dumbbell_var}
\end{figure}

\subsection{Unit ball example}
Next, we use a three dimensional example from \cite{masc:hwan:2003}, with $\Omega=B_3(\bszero,1)$. The BVP has $\Delta u(\bsz)=0$ for $z\in \Omega$ with boundary condition
$u(\bsz)=[(z_1-2)^2+z_2^2+z_3^2]^{-1/2}$ for $\bsz\in\partial\Omega$.
The analytic solution has the same formula as the boundary condition. 
To generate a uniform sample on the sphere, we use the hatbox transformation~\eqref{eq:hatbox} described in Section~\ref{sec:wosinrqmc}

The hatbox transformation $\psi_0$ is not Lipschitz, so the analysis in Section \ref{sec:main} does not apply to this update. Figure~\ref{fig:ex3_naive_rmse} shows an apparent RQMC  variance of $O(n^{-1.14})$ which we take as evidence that our sufficient conditions are not necessary conditions.

\begin{figure}[t]
    \centering
    \includegraphics[width=0.8\linewidth]{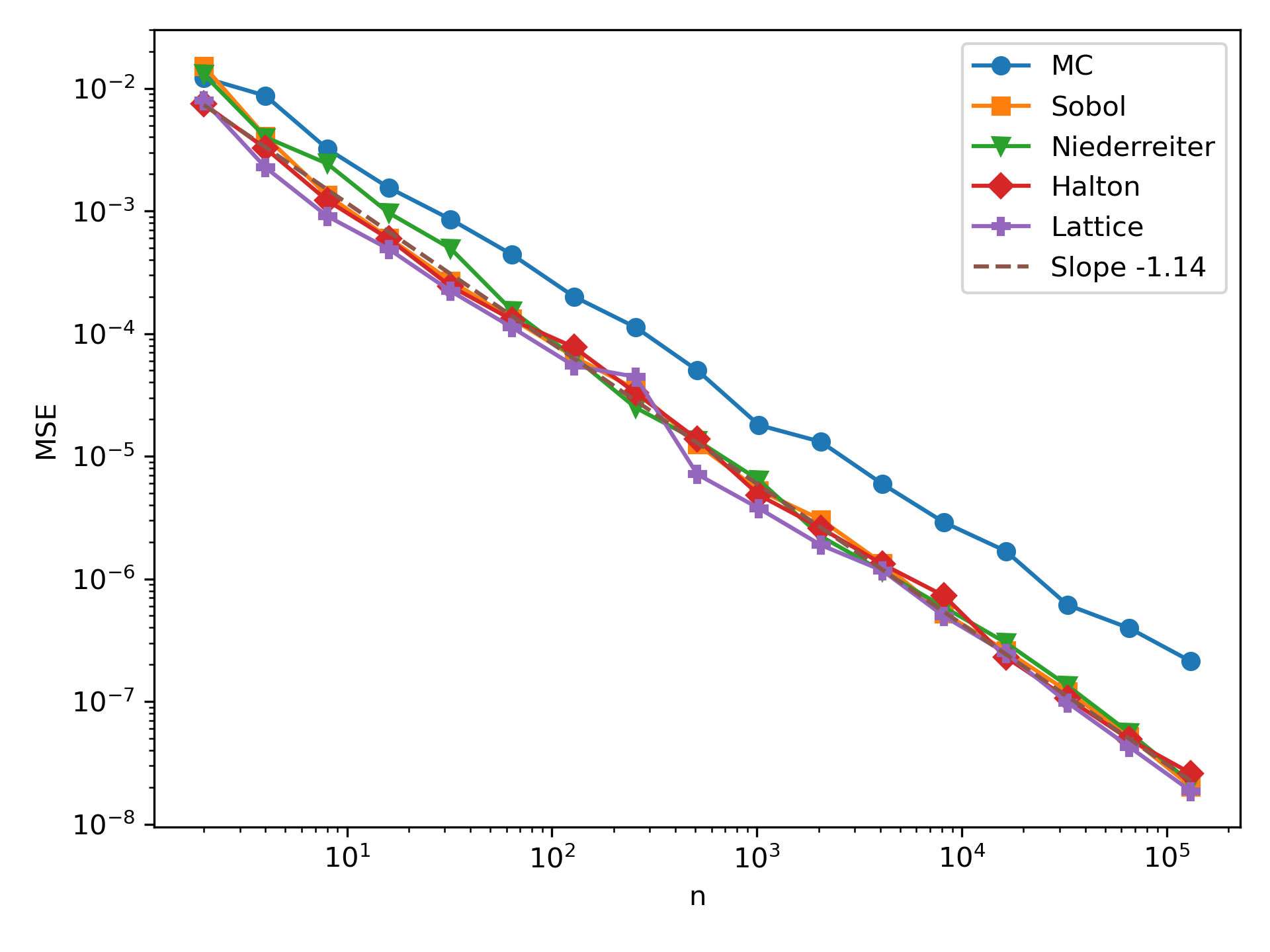}
    \caption{MSE of the standard WoS estimator for the unit ball example, for $z_0=(0.2, 0.3, -0.1)$ and $\varepsilon=10^{-4}$.}
    \label{fig:ex3_naive_rmse}
\end{figure}

\subsection{Summary of examples}\label{sec:summary}
\begin{table}[t]
\centering
\setlength{\tabcolsep}{3.5pt}
\renewcommand{\arraystretch}{1.08}
\small
\begin{tabular}{lcccccccccc}
\toprule
& \multicolumn{2}{c}{Sobol} & \multicolumn{2}{c}{Lattice} & \multicolumn{2}{c}{Halton} & \multicolumn{2}{c}{Niederreiter} & \multicolumn{2}{c}{MC} \\
\cmidrule(lr){2-3}\cmidrule(lr){4-5}\cmidrule(lr){6-7}\cmidrule(lr){8-9}\cmidrule(lr){10-11}
Example & $\beta$ & $\alpha$ & $\beta$ & $\alpha$ & $\beta$ & $\alpha$ & $\beta$ & $\alpha$ & $\beta$ & $\alpha$ \\
\midrule
  Gasket & $-$1.10 & $\phm$5.78 & $-$1.08 & $\phm$5.81 & $-$1.11 & $\phm$5.98 & $-$1.14 & $\phm$6.13 & $-$1.01 & $\phm$6.41 \\
  Unit disk & $-$1.04 & $-$3.90 & $-$1.15 & $-$3.03 & $-$1.13 & $-$3.27 & $-$1.12 & $-$3.28 & $-$1.01 & $-$2.31 \\
  Dumbbell & $-$1.08 & $-$3.87 & $-$1.16 & $-$3.24 & $-$1.02 & $-$4.28 & $-$1.10 & $-$3.60 & $-$1.02 & $-$3.38 \\
  Pac-Man & $-$1.09 & $-$2.47 & $-$1.00 & $-$3.15 & $-$1.05 & $-$2.77 & $-$1.17 & $-$0.89 & $-$0.98 & $-$2.52 \\
  Unit ball & $-$1.16 & $-$3.99 & $-$1.15 & $-$4.26 & $-$1.15 & $-$3.99 & $-$1.12 & $-$4.27 & $-$1.00 & $-$3.71 \\
\bottomrule
\end{tabular} 
\caption{Regression coefficients for $\log(\widehat{\var}(\hat\mu))\doteq\alpha+\beta \log(n)$ fit to WoS methods using $n\ge 128$. 
}
\label{tab:rqmc_slopes_intercepts}
\end{table}

Table~\ref{tab:rqmc_slopes_intercepts} shows regression fits of log variance (or log MSE when available) versus $\log(n)$ for our sampling methods and examples. Most of the curves looked to be linear with a little noise, but some of them showed nonlinearities at small $n$. We used least squares on the data with $7\le\log_2(n)\le 17$.  The Niederreiter points had quite bad performance for the Pac-Man example. While they had the best slope they had a very large constant term.

Table~\ref{tab:vrf_17} computes variance reduction factors corresponding to the slopes and intercepts in Table~\ref{tab:rqmc_slopes_intercepts}. The variance reduction factors are modest though it is clear that they improve somewhat for larger $n$.

\begin{table}[t]
\centering
\sisetup{table-format=2.1, table-number-alignment=center}
\begin{tabular}{l S S S S}
\toprule
Example & {Sobol} & {Lattice} & {Halton} & {Niederreiter} \\
\midrule
Gasket    &  5.4 &  4.2 &  5.0 &  6.1 \\
Unit disk &  7.0 & 10.7 & 10.7 & 9.6 \\
Dumbbell  &  3.3 &  4.5 &  2.5 &  3.2 \\
Pac-Man    &  3.5 &  2.4 &  2.9 &  1.8 \\
Unit ball &  8.7 & 10.2 &  7.7 &  7.2 \\
\bottomrule
\end{tabular}
\caption{Estimated variance reduction factors at $n=2^{17}$.}
\label{tab:vrf_17}
\end{table}



\subsection{Alternative samplers}

For $\Omega\subset\real^2$, the position $\bsz_k$ is a periodic function on $[0,1)^k$, which is an advantage for the lattice sampling methods.  The decision to stop or not to stop at step $k$ introduces a discontinuity which violates the smoothness that lattice samplers need for high accuracy.  Additionally, the function $u$ is not necessarily a smooth function on $\partial\Omega$.

We ran the algorithm with fixed large values of $k$ which then give small $\dist(\bsz_k,\partial\Omega)$ with high probability. This strategy has been used by \cite{wu:etal:2025}.  
After the $k$ steps, we project $\bsz_k$ to $\partial\Omega$ and evaluate $h$ at that projected point. This removes the discontinuity due to the stopping rule, leaving a fixed-step integrand based on the Lipschitz map $x\mapsto \bsz_k(x)$. We did not see the lattice sampler outperform other RQMC methods for this algorithm except for very small values, such as $k=3$, where the bias is too large for the method to be useful.  We considered the unit disk case taking $k=20$ steps toward the boundary, where $h(\bsz)$ is a smooth function. This is the smoothest example we considered and the lattice rule did not show an advantage over the other RQMC methods.

A harmonic function has $u(\bsz_0)=\e(u(\bsz))$ for $\bsz\sim\runif(S(\bsz_0,r))$. By averaging over radii, the same is true for $\bsz\sim\runif(B(\bsz_0,r))$. \cite{wu:etal:2025}   include an algorithm where the first step is uniform over $B(\bsz_0,r_0)$ to make the estimate a smoother function of $\bsz_0$. We investigated such steps but they did not produce an algorithm that made lattice methods dominant.

\section{Discussion}\label{sec:discussion}

We have replaced MC sampling by RQMC sampling in the WoS algorithm for a collection of BVP problems in $\real^2$ and $\real^3$. We see rates that are close to $n^{-1-1/\tilde k}$ for some `effective' dimension $\tilde k$ over values of $n$ up to $2^{17}$. The actual error has contributions from many values of $k$, and there is likely to be a nonzero contribution from $k=K$, leading to a rate of only $O(n^{-1-1/K})$, but there is no evidence that such a pessimistic rate will be relevant for sample sizes near the ones we have used.
The unit disk example has a very smooth boundary function and illustrates the theorem from \cite{liu:2025}.

The Pac-Man example uses $3k$ uniform variables to take $k$ steps because it also samples disks at each step.  The inferior convergence rate we saw for it could be due to this increased dimension.  The integrand is only discontinuous in $k$ of those variables used to define whether the walk terminates in $k$ steps. 

\section*{Materials and methods}

We used ChatGPT, Claude and Gemini in parts of this work. The versions we used were from the early months of 2026.  Those AIs helped us find relevant literature much more quickly than we could have done with a search engine. Many suggested references were not actually useful and had to be discarded but others were critical.  The AIs found weaknesses in some early versions of our proofs and made suggestions.  Quite often the AI suggestions were naive about subtleties such as the varying definitions of rectifiability and we rejected many AI suggestions.   On the other hand, it was an AI that noticed that the second claim in Theorem~\ref{thm:itswosregular} follows from the same argument as the first claim.  The AIs made some suggestions about writing style and we adopted a few of those for clarity but we wrote everything ourselves.  AI found some typos and incomplete sentences that we corrected.   
They also generated, following our instructions, the Python code used to draw the gasket and dumbbell domains shown in Figure~\ref{fig:gasket} and \ref{fig:dumbbell} respectively, and the R code used to produce Figure~\ref{fig:twosurfaces}, which illustrates the chambers. An AI suggested the term `chamber'. The AIs also pointed us toward useful theorems on manifolds, after which we identified versions that fit the paper without requiring substantial additional background. When we were investigating an apparent sign error in our computations for the dumbbell example, an AI recognized that \cite{sawh:mill:gkio:cran:2023} uses the negative semidefinite Laplacian. The AIs also assisted with LaTeX syntax in resolving formatting issues and with improving several BibTeX entries.

\section*{Acknowledgments}

We thank Rohan Sawhney and Yang Liu for some helpful comments.
The second author visited Bob Carpenter's group at the Flatiron Institute and had productive discussions there with Misha Padidar,  Michael Czekanski, Dan Fortunato and Charles Epstein. 
\bibliographystyle{chicago}
\bibliography{wos}

\end{document}